\UseRawInputEncoding
\documentclass[11pt, reqno]{amsart}
\usepackage{amsfonts,latexsym,enumerate}
\usepackage{amsmath}
\usepackage{amscd}
\usepackage{float,amsmath,amssymb,mathrsfs,bm,multirow,graphics}
\usepackage[dvips]{graphicx}
\usepackage[percent]{overpic}
\usepackage{amsaddr}
\usepackage{todonotes}

\newcommand{\nequation}{\setcounter{equation}{0}}

\newcommand{\R}{{\Bbb R}}

\newcommand{\C}{{\Bbb C}}

\newcommand{\proofbegin}{\noindent{\it Proof.\quad}}
\newcommand{\proofend}{\hfill$\Box$\bigskip}

\newcommand{\diag}{\text{\upshape diag\,}}
\newcommand{\re}{\text{\upshape Re\,}}

\usepackage{tikz}
\usepackage{pict2e}
\usetikzlibrary{arrows,calc,chains, positioning, shapes.geometric,shapes.symbols,decorations.markings,arrows.meta}
\usetikzlibrary{patterns,angles,quotes}
\usetikzlibrary{decorations.markings}
\tikzset{middlearrow/.style={
			decoration={markings,
				mark= at position 0.6 with {\arrow{#1}} ,
			},
			postaction={decorate}
		}
	}
\tikzset{->-/.style={decoration={
				markings,
				mark=at position #1 with {\arrow{latex}}},postaction={decorate}}}
	
\tikzset{-<-/.style={decoration={
				markings,
				mark=at position #1 with {\arrowreversed{latex}}},postaction={decorate}}}
				
				\tikzset{
	master/.style={
		execute at end picture={
			\coordinate (lower right) at (current bounding box.south east);
			\coordinate (upper left) at (current bounding box.north west);
		}
	},
	slave/.style={
		execute at end picture={
			\pgfresetboundingbox
			\path (upper left) rectangle (lower right);
		}
	}
}
\tikzset{cross/.style={cross out, draw, 
         minimum size=2*(#1-\pgflinewidth), 
         inner sep=0pt, outer sep=0pt}}

\def\XXint#1#2#3{{\setbox0=\hbox{$#1{#2#3}{\int}$}
\vcenter{\hbox{$#2#3$}}\kern-.5\wd0}}

\allowdisplaybreaks

\newtheorem{theorem}{Theorem}[section]
\newtheorem{proposition}[theorem]{Proposition}
\newtheorem{corollary}[theorem]{Corollary}
\newtheorem{lemma}[theorem]{Lemma}
\newtheorem{definition}[theorem]{Definition}
\newtheorem{assumption}[theorem]{Assumption}
\newtheorem{remark}[theorem]{Remark}

\newtheorem{RHproblem}[theorem]{RH problem}
\newtheorem{figuretext}{Figure}

\usepackage[colorlinks=true]{hyperref}
\hypersetup{urlcolor=blue, citecolor=red, linkcolor=blue}

\input epsf
\title[The ``good" Boussinesq equation on the half-line]
{The ``good" Boussinesq equation on the half-line: a Riemann-Hilbert approach}

\author{C\MakeLowercase{hristophe} C\MakeLowercase{harlier$^{1}$ and} J\MakeLowercase{onatan} L\MakeLowercase{enells$^{2}$}}
\address{\hspace{0.35cm}$^{1}$Research Institute in Mathematics and Physics, UCLouvain, Belgium. 
\newline
$^{2}$Department of Mathematics, KTH Royal Institute of Technology, Sweden.}
\email{christophe.charlier@uclouvain.be, jlenells@kth.se}

\begin{document}

\begin{abstract} 
We consider the ``good" Boussinesq equation on the half-line. Assuming existence of the solution, we prove that it can be recovered from the solution of a $3\times 3$ Riemann-Hilbert problem that depends only on the initial and boundary values, and whose jump contour consists of twelve half-lines.
\end{abstract}

\maketitle

\noindent
{\small{\sc AMS Subject Classification (2020)}: 35G31, 35Q15, 37K15.}

\noindent
{\small{\sc Keywords}: Unified transform method, Riemann-Hilbert problem, inverse scattering transform, initial-boundary value problem.}


\section{Introduction}\nequation

In 1872, Boussinesq \cite{B1872} derived the equation 
\begin{align}\label{badboussinesq}
u_{tt} = u_{xx} + (u^2)_{xx} + u_{xxxx}
\end{align}
as a model for long waves in shallow water. The function $u(x,t)$ is real-valued and subscripts denote partial derivatives. Equation (\ref{badboussinesq}) is linearly ill-posed and is therefore commonly referred to as the ``bad'' Boussinesq equation.

By changing the sign of the $u_{xxxx}$ term in \eqref{badboussinesq}, and then replacing $u$ by $-u$, we obtain the ``good'' Boussinesq equation
\begin{align}\label{goodboussinesq}
u_{tt} = u_{xx} - (u^2)_{xx} - u_{xxxx},
\end{align}
which is linearly well-posed. Equation (\ref{goodboussinesq}) models nonlinear vibrations along a string \cite{Z1974} and is therefore also known as the ``nonlinear string equation'' \cite{FST1983}. Well-posedness results  can be found in e.g. \cite{BS1988, CT2017, F2009, L1993} for (\ref{goodboussinesq}) on the line, and in \cite{HM2015} for (\ref{goodboussinesq}) on the half-line.

By replacing $u$ by $u - \frac{1}{2}$ and $u + \frac{1}{2}$ in (\ref{badboussinesq}) and (\ref{goodboussinesq}), respectively, we obtain the equations
\begin{align}\label{boussinesqnouxx sigma}
  u_{tt}  + \sigma((u^2)_{xx} + u_{xxxx}) = 0, \qquad \sigma = \pm 1,
\end{align}
where $\sigma =1$ and $\sigma = -1$ correspond to the good and bad versions, respectively. Results on inverse scattering transforms have been obtained in the case of Schwartz class initial data for three out of the four aforementioned equations, namely equations \eqref{badboussinesq} \cite{CLmain}, \eqref{boussinesqnouxx sigma} with $\sigma = -1$ \cite{DTT1982}, and \eqref{boussinesqnouxx sigma} with $\sigma = 1$ \cite{CLgoodboussinesq}. Using the Riemann-Hilbert approaches developed in \cite{CLmain, CLgoodboussinesq}, the long-time behavior of the solution is obtained in \cite{CLmain} for \eqref{badboussinesq}, and in \cite{CLWasymptotics} for \eqref{boussinesqnouxx sigma} with $\sigma = 1$. All the results contained in \cite{CLgoodboussinesq, CLmain, CLWasymptotics, DTT1982} concern initial-value problems. Inverse scattering results for \eqref{badboussinesq} on the half-line were obtained in \cite{C H-L}.

In this paper, we consider the initial-boundary value problem for the ``good" Boussinesq equation \eqref{boussinesqnouxx sigma} with $\sigma =1$ on a half-line, i.e., for
\begin{align}\label{boussinesqnouxx}
& u_{tt} + (u^2)_{xx} + u_{xxxx} = 0.
\end{align}

It will convenient for us to consider the following equation, which can be obtained from \eqref{boussinesqnouxx} by rescaling the $x$ and $t$ coordinates:
\begin{align}\label{boussinesq}
& u_{tt} + \frac{4}{3} (u^2)_{xx} + \frac{1}{3} u_{xxxx} = 0.
\end{align}
Following \cite{Z1974, DTT1982, CLgoodboussinesq}, we rewrite (\ref{boussinesq}) as the system
\begin{align}\label{boussinesqsystem}
& \begin{cases}
 v_{t} + \frac{1}{3}u_{xxx} + \frac{4}{3}(u^{2})_{x} = 0,
 \\
 u_t = v_x,
\end{cases}
\end{align}
which is equivalent to (\ref{boussinesq}) provided that the initial data $u_1(x) := u_t(x,0)$ satisfy
\begin{align}\label{u1zeromean}
\int_\R u_1(x) dx = 0,
\end{align}
see \cite[Lemma 2.8]{CLgoodboussinesq} for details. To be more precise, we consider the initial-boundary value problem for \eqref{boussinesqsystem} in the half-line domain
\begin{align}\label{halflinedomain}
\{0 < x < \infty, \; 0 < t < T\},
\end{align}
where $T > 0$ is a finite constant, and with the initial data 
\begin{align}\label{initial data}
u_0(x) := u(x,0), \qquad v_0(x) := v(x,0), \qquad x \geq 0,
\end{align}
as well as the boundary values
\begin{align}\label{boundaryvalues}
\tilde{u}_{0}(t) := u(0,t), \quad \tilde{u}_{1}(t) := u_{x}(0,t), \quad \tilde{u}_{2}(t) := u_{xx}(0,t), \quad \tilde{v}_0(t) := v(0,t), \qquad t \in [0,T].
\end{align}

Assuming that the solution of the initial-boundary value problem for (\ref{boussinesqsystem}) exists, we will show that it can be expressed in terms of the solution of a $3 \times 3$ matrix Riemann-Hilbert (RH) problem whose jump matrix is expressed in terms of four reflection coefficients $\{r_j(k)\}_{j=1}^{4}$. For conciseness, we restrict ourselves in this paper to smooth solutions with rapid decay as $x \to +\infty$. 

Our main results are presented in the form of two theorems as follows: 
\begin{enumerate}[$-$]
\item Theorem \ref{r1r2th} establishes several properties of the reflection coefficients $\{r_j(k)\}_{j=1}^{4}$. 

\item Theorem \ref{RHth} solves the inverse problem of recovering the solution $\{u(x,t), v(x,t)\}$ of (\ref{boussinesqsystem}) from the reflection coefficients $\{r_j(k)\}_{j=1}^{4}$. The recovery involves the solution $M(x,t,k)$ of a $3 \times 3$ matrix RH problem whose jump contour $\Gamma$ consists of the six lines $\R \cup \omega \R \cup \omega^2 \R \cup i \R \cup \omega i \R \cup \omega^{2} i \R$ where $\omega := e^{2\pi i/3}$, see Figure \ref{fig: Dn}, and whose jump matrix is given explicitly in terms of $\{r_j(k)\}_{j=1}^{4}$, see (\ref{vdef}).
\end{enumerate}
The above theorems are formulated for the system (\ref{boussinesqsystem}). Hence our results also apply to equation (\ref{boussinesq}) provided that $u$ satisfies (\ref{u1zeromean}). 
Moreover, if $u$ satisfies (\ref{boussinesq}) then 
\begin{align}\label{fromboussinesqtogoodboussinesq}
  \tilde{u}(x,t) = \frac{4}{\sqrt{3}}u\bigg(\frac{x}{3^{1/4}},t\bigg) + \frac{1}{2}
\end{align}
satisfies (\ref{goodboussinesq}). As a consequence, we also obtain results for equation (\ref{goodboussinesq}) for solutions approaching $1/2$ as $x \to +\infty$. 

The proofs of our main results use the machinery for the analysis of initial-boundary value problems for integrable evolution equations with $3\times 3$ Lax pairs developed in \cite{L3x3, L2013}, which in turn is based on Fokas's unified transform method \cite{F1997, F2002}. Applications of Fokas's method to linear PDEs and to integrable nonlinear PDEs with $2 \times 2$ Lax pairs can be found in \cite{
BFS2004, BS2003, CFK2025, DFL2021, DTV2014, FIS2005, FL2010, JGM2021, I2003, K2010, P2004, PP2010}.

\begin{figure}
\begin{center}
\begin{tikzpicture}
\node at (0,0) {};
\draw[black,line width=0.45 mm,->-=0.7] (0,0)--(0:4);
\draw[black,line width=0.45 mm,->-=0.7] (0,0)--(30:4);
\draw[black,line width=0.45 mm,->-=0.7] (0,0)--(60:4);
\draw[black,line width=0.45 mm,->-=0.7] (0,0)--(90:4);
\draw[black,line width=0.45 mm,->-=0.7] (0,0)--(120:4);
\draw[black,line width=0.45 mm,->-=0.7] (0,0)--(150:4);
\draw[black,line width=0.45 mm,->-=0.7] (0,0)--(180:4);
\draw[black,line width=0.45 mm,->-=0.7] (0,0)--(-30:4);
\draw[black,line width=0.45 mm,->-=0.7] (0,0)--(-60:4);
\draw[black,line width=0.45 mm,->-=0.7] (0,0)--(-90:4);
\draw[black,line width=0.45 mm,->-=0.7] (0,0)--(-120:4);
\draw[black,line width=0.45 mm,->-=0.7] (0,0)--(-150:4);

\draw[black,line width=0.15 mm] ([shift=(0:1cm)]0,0) arc (0:30:1cm);

\node at (15:3.1) {\small $D_{1}$};
\node at (45:3.1) {\small $D_{2}$};
\node at (75:3.1) {\small $D_{3}$};
\node at (105:3.1) {\small $D_{4}$};
\node at (135:3.1) {\small $D_{5}$};
\node at (165:3.1) {\small $D_{6}$};
\node at (195:3.1) {\small $D_{7}$};
\node at (225:3.1) {\small $D_{8}$};
\node at (255:3.1) {\small $D_{9}$};
\node at (285:3.1) {\small $D_{10}$};
\node at (315:3.1) {\small $D_{11}$};
\node at (345:3.1) {\small $D_{12}$};

\node at (1.4,0.3) {\small $\pi/6$};

\end{tikzpicture}
\end{center}
\begin{figuretext}\label{fig: Dn}
The contour $\Gamma$ and the open sets $D_{n}$, $n=1,\ldots,12$.
\end{figuretext}
\end{figure}

\section{Main results}\label{mainsec}
Our first result concerns the direct problem of constructing  $\{r_j(k)\}_{j=1}^4$ in terms of the initial and boundary values $\{u_0, v_0, \tilde{u}_{0}, \tilde{u}_{1}, \tilde{u}_{2}, \tilde{v}_0\}$ appearing in \eqref{initial data}--\eqref{boundaryvalues}.

\subsection{The direct problem}
Let $\R_{+} := [0,+\infty)$ and let $\mathcal{S}(\R_{+})$ denote the Schwartz class of smooth functions on $\R_+$ that are rapidly decaying at $+\infty$ (but not necessarily at $0$).
Let $T < +\infty$. Let $u_0, v_0 \in \mathcal{S}(\R_{+})$ be two real-valued functions and let $\tilde{u}_{0}, \tilde{u}_{1}, \tilde{u}_{2}, \tilde{v}_0 \in C^{\infty}([0,T])$. These functions are assumed to be such that there exists a solution $\{u(x,t),v(x,t)\}$ of \eqref{boussinesqsystem} for all $x\geq 0$ and $t\in [0,T]$ satisfying \eqref{initial data} and \eqref{boundaryvalues}. The associated reflection coefficients $\{r_j(k)\}_{j=1}^{4}$ are defined as follows (see Section \ref{specsec} for more details).

Let $\omega := e^{\frac{2\pi i}{3}}$ and define $\{l_j(k), z_j(k)\}_{j=1}^3$ by
\begin{align}\label{lmexpressions intro}
&l_j(k) = \omega^j k, \quad z_j(k) = \omega^{2j} k^{2}, \qquad k \in \C.
\end{align}
For $x\geq 0$, $t\in [0,T]$, and $k\in \C\setminus \{0\}$, let $\mathsf{U}(x,t,k)$ and $\mathsf{V}(x,t,k)$ be given by
\begin{align}\label{mathsfUdef intro}
& \mathsf{U}(x,t,k) = P(k)^{-1} \begin{pmatrix}
0 & 0 & 0 \\
0 & 0 & 0 \\
-u_{x}(x,t)-v(x,t)  & -2u(x,t) & 0
\end{pmatrix} P(k), \\
& \mathsf{V}(x,t,k) = P(k)^{-1} \begin{pmatrix}
\frac{4}{3}u(x,t) & 0 & 0 \\[0.1cm]
\frac{u_{x}(x,t)}{3}-v(x,t) & -\frac{2}{3}u(x,t) & 0 \\[0.1cm]
\frac{u_{xx}(x,t)}{3}-v_{x}(x,t) & - \frac{u_{x}(x,t)}{3}-v(x,t) & - \frac{2}{3}u(x,t)
\end{pmatrix} P(k),
\end{align} 
where
\begin{align}\label{Pdef intro}
P(k) = \begin{pmatrix}
\omega & \omega^{2} & 1  \\
\omega^{2} k & \omega k & k \\
k^{2} & k^{2} & k^{2}
\end{pmatrix}.
\end{align}
Define the $3 \times 3$-matrix valued eigenfunctions $\mu_1(0,t,k), \mu_3(x,t,k), \mu^{A}_1(0,t,k), \mu^{A}_3(x,t,k)$ as the unique solutions of the Volterra integral equations
\begin{subequations}\label{mumuAdef intro}
\begin{align}
& \mu_1(0,t,k) = I -  \int_{t}^{T} e^{(t-t')\hat{\mathcal{Z}}(k)} (\mathsf{V}\mu_1)(0,t',k) dt',  \label{mu1def intro} \\
& \mu_3(x,t,k) = I -  \int_{x}^{+\infty} e^{(x-x')\hat{\mathcal{L}}(k)} (\mathsf{U}\mu_3 ) (x',t,k)dx',  \label{mu3def intro} \\
& \mu^{A}_1(0,t,k) = I +  \int_{t}^{T} e^{(t'-t)\hat{\mathcal{Z}}(k)} (\mathsf{V}^{T}\mu^{A}_{1})(0,t',k) dt',  \label{muA1def intro} \\
& \mu^{A}_3(x,t,k) = I + \int_{x}^{+\infty} e^{(x'-x)\hat{\mathcal{L}}(k)} (\mathsf{U}^{T}\mu^{A}_{3})(x',t,k) dx',  \label{muA3def intro}
\end{align}
\end{subequations}
where $\mathcal{L} := \diag(l_1 , l_2 , l_3)$, $\mathcal{Z} := \diag(z_1 , z_2 , z_3)$, $\mathsf{U}^{T}$ and $\mathsf{V}^{T}$ are the transposes of $\mathsf{U}$ and $\mathsf{V}$, respectively, and given a $3\times 3$ matrix $B$, $e^{\hat{B}}$ denotes the operator which acts on a $3 \times 3$ matrix $A$ by $e^{\hat{B}}A = e^B A e^{-B}$. Note that the definitions of $\mu_1(0,t,k), \mu_3(x,0,k), \mu^{A}_1(0,t,k), \mu^{A}_3(x,0,k)$ depend on $\{u(x,t),v(x,t)\}$ only through the initial and boundary values. Define also $s,s^A,S,S^{A}$ by 
\begin{align}
& S(k) = I-\int_{0}^{T}e^{-t\hat{\mathcal{Z}}(k)}(\mathsf{V}\mu_{1})(0,t,k)dt, & & s(k) = I-\int_{0}^{+\infty}e^{-x\hat{\mathcal{L}}(k)}(\mathsf{U}\mu_{3})(x,0,k)dx, \label{def of Ss} \\
& S^{A}(k) = I+\int_{0}^{T}e^{t\hat{\mathcal{Z}}(k)}(\mathsf{V}^{T}\mu_{1}^{A})(0,t,k)dt, & & s^{A}(k) = I+\int_{0}^{+\infty}e^{x\hat{\mathcal{L}}(k)}(\mathsf{U}^{T}\mu_{3}^{A})(x,0,k)dx, \label{def of SAsA}
\end{align}
and set
\begin{align}\label{lol1}
s^L(k) := (S^A(k))^T s(k).
\end{align}


The four spectral functions $\{r_j(k)\}_{j=1}^4$ are defined by
\begin{align}\label{r1r2def}
\begin{cases}
r_1(k) = \frac{s^L_{12}}{s^L_{11}}, & k \in \Gamma_{1}=(0,\infty),
	\\[0.2cm]
r_2(k) = \frac{s^A_{12}S^{A}_{33}-s^{A}_{32}S^{A}_{13}}{s^A_{11}S^{A}_{33}-s^{A}_{31}S^{A}_{13}}, \quad & k \in \Gamma_{7}=(-\infty,0), \\[0.2cm]
r_{3}(k) = \frac{S^{A}_{31}}{s^{A}_{33}} \frac{1}{s^L_{11}}, \quad & k \in \Gamma_{2}=e^{\frac{\pi i}{6}}(0,\infty), \\[0.2cm]
r_{4}(k) = \frac{S^{A}_{31}}{s^{A}_{33}} \frac{s_{21}}{s^{A}_{33}S^{A}_{11}-s^{A}_{13}S^{A}_{31}}, \quad & k \in \Gamma_{2}=e^{\frac{\pi i}{6}}(0,\infty).
\end{cases}
\end{align}

\subsubsection{Assumption of no solitons}
Solitons correspond to the presence of zeros for certain spectral functions. For simplicity, we assume throughout the paper that no solitons are present. 
\begin{assumption}[Absence of solitons]\label{solitonlessassumption}\upshape
Assume that
\begin{itemize}
\item $(S^{-1}s)_{11}$ is nonzero on $(\bar{D}_{1}\cup \bar{D}_{2}) \setminus\{0\}$,
\item $s^{A}_{33}$ is nonzero on $\bar{D}_{1}\setminus\{0\}$,
\item $s^{A}_{33}S^{A}_{11}-s^{A}_{13}S^{A}_{31}$ is nonzero on $\bar{D}_{2}\setminus\{0\}$.
\end{itemize}
\end{assumption}

\subsubsection{Assumption of generic behavior at $k = 0$}
We will also make the following assumption, which holds true for generic initial and boundary values.

\begin{assumption}[Generic behavior at $k = 0$]\label{originassumption}\upshape
Assume that
\begin{align*}
& \lim_{k \to 0} k^2 s(k)_{33} \neq 0, && \lim_{k \to 0} k^2 s^A(k)_{33} \neq 0, \\
& \lim_{k \to 0} k^2 S(k)_{33} \neq 0, && \mathscr{S}^{A(-2)} := \lim_{k \to 0} k^2 S^A(k)_{33} \neq 0, \\
& \lim_{k \to 0} k \big( S^A(k)_{33} - k^{-2}\mathscr{S}^{A(-2)} \big) \neq 0, &&  \lim_{k \to 0} k \big( S^A(k)_{32} - k^{-2}\mathscr{S}^{A(-2)} \big) \neq 0,
	\\
& \lim_{k\to 0} k^2 (S(k)^{-1}s(k))_{11} 
 \neq 0, 
&&
\lim_{k\to 0} k^3\big( s^A(k)_{33} S^A(k)_{11} - s^A(k)_{13} S^A(k)_{31} \big)
\neq 0.
\end{align*}
\end{assumption}

\subsubsection{Statement of the first theorem}
We can now state our first theorem, which concerns the direct problem, that is, the map from $\{u_0, v_0, \tilde{u}_{0}, \tilde{u}_{1}, \tilde{u}_{2}, \tilde{v}_0\}$ to $\{r_j\}_{j=1}^4$. 
The proof is presented in Section \ref{r1r2subsec}.

\begin{theorem}[Properties of $r_j(k)$, $j=1,2,3,4$]\label{r1r2th}
Suppose $u_0,v_0 \in \mathcal{S}(\R_{+})$ and $\tilde{u}_{0}, \tilde{u}_{1}, \tilde{u}_{2}, \tilde{v}_0 \in C^{\infty}([0,T])$ are such that Assumption \ref{solitonlessassumption} holds. Then the spectral functions $r_1:(0,\infty) \to \C$, $r_2:(-\infty,0) \to \C$ and $r_{3},r_{4}:\Gamma_{2}\to \C$ are well-defined by (\ref{r1r2def}) and have the following properties:
\begin{enumerate}[$(i)$]
 \item $r_1 \in C^\infty((0,\infty))$, $r_2 \in C^\infty((-\infty,0))$ and $r_{3},r_{4}\in C^{\infty}(\Gamma_{2})$. 
 
 \item The functions $r_1(k)$, $r_2(k)$, $r_3(k)$, $r_4(k)$, and their derivatives $\partial_{k}^{j}r_{l}(k)$ have continuous boundary values at $k=0$ for $l = 1,2,3,4$ and for all $j = 0,1,2,\ldots$, and there exist expansions
 \begin{subequations}\label{r1r2atzero}
\begin{align}
& r_{1}(k) = r_{1}(0) + r_{1}'(0)k + \tfrac{1}{2}r_{1}''(0)k^{2} + \cdots, & & k \to 0, \ k >0, \\
& r_{2}(k) = r_{2}(0) + r_{2}'(0)k + \tfrac{1}{2}r_{2}''(0)k^{2} + \cdots, & & k \to 0, \ k <0, \\
& r_{3}(k) = r_{3}(0) + r_{3}'(0)k + \tfrac{1}{2}r_{3}''(0)k^{2} + \cdots, & & k \to 0, \ k \in \Gamma_{2}, \\
& r_{4}(k) = r_{4}(0) + r_{4}'(0)k + \tfrac{1}{2}r_{4}''(0)k^{2} + \cdots, & & k \to 0, \ k \in \Gamma_{2},
\end{align}
\end{subequations}
which can be differentiated termwise any number of times.

\item Suppose moreover that Assumption \ref{originassumption} holds. Then the leading coefficients satisfy
\begin{align}\label{r1r2at0}
r_{1}(0) = \omega, \quad r_{2}(0) = 1, \quad r_{3}(0)=r_{3}'(0)=r_{4}(0)=0, \quad r_{3}''(0)\neq 0 \neq r_{4}'(0).
\end{align}
In particular, as $k\to 0$, $k\in \Gamma_{2}$, $r_{3}(k)$ vanishes quadradically and $r_{4}(k)$ vanishes linearly.

\item There exist expansions
 \begin{subequations}\label{r1r2r3r4 slow decay}
\begin{align}
& r_{1}(k) = r_{1}^{(1)} k^{-1} + r_{1}^{(2)}k^{-2} + \cdots, & & k \to \infty, \ k >0, \\
& r_{2}(k) = r_{2}^{(1)} k^{-1} + r_{2}^{(2)}k^{-2} + \cdots, & & k \to \infty, \ k <0, \\
& r_{3}(k) = r_{3}^{(1)} k^{-1} + r_{3}^{(2)}k^{-2} + \cdots, & & k \to \infty, \ k \in \Gamma_{2}, \\
& r_{4}(k) = r_{4}^{(2)} k^{-2} + r_{4}^{(3)}k^{-3} + \cdots, & & k \to \infty, \ k \in \Gamma_{2},
\end{align}
\end{subequations}
which can be differentiated termwise any number of times. Moreover, 
\begin{align}\label{lol3}
r_{4}^{(2)} = \big(\overline{r_{1}^{(1)}}-\omega^{2}r_{3}^{(1)}\big)r_{3}^{(1)}.
\end{align}

\item
 If $(S^{T}s^{A})_{33}$ is nonzero on $(0,+\infty)$, then $|r_{1}(k)|<1$ for all $k > 0$.
 If $S^{A}_{33}$ is nonzero on $(-\infty,0)$, then $|r_{2}(k)| <1$ for all $k < 0$.

\end{enumerate} 
\end{theorem}

\subsection{The inverse problem}
We next consider the inverse problem of recovering $\{u, v\}$ from the scattering data. Since we are assuming that no solitons are present, the scattering data consists only of the four reflection coefficients $\{r_{j}(k)\}_{j=1}^{4}$. We will show that the inverse problem can be solved by means of a RH problem for a $3 \times 3$-matrix valued function $M$ whose jump matrix is expressed in terms of $\{r_{j}(k)\}_{j=1}^{4}$.

\subsubsection{The RH problem for $M$}
Let $\Gamma$ be the contour consisting of the six lines $\R \cup \omega \R \cup \omega^2 \R \cup i\R \cup \omega i\R \cup \omega^2 i\R$ oriented away from the origin as in Figure \ref{fig: Dn}.
For $1 \leq i \neq j \leq 3$, define $\theta_{ij} \equiv \theta_{ij}(x,t,k)$ by
$$\theta_{ij}(x,t,k) = (l_i - l_j)x + (z_i - z_j)t.$$
Given a function $f(k)$ of $k\in \C$, we let $f^{*}$ denote the Schwartz conjugate of $f$, i.e., $f^{*}(k) = \overline{f(\overline{k})}$. Define the jump matrix $v(x,t,k)$ for $k \in \Gamma$ by
\begin{align}
&  v_1 = \begin{pmatrix}  
 1 & - r_1(k)e^{-\theta_{21}} & 0 \\
  r_1^*(k)e^{\theta_{21}} & 1 - |r_1(k)|^2 & 0 \\
  0 & 0 & 1
  \end{pmatrix}, &&  v_2 = \begin{pmatrix}   
 1 & 0 & -r_{3}(k)e^{-\theta_{31}} \\
 0 & 1 & r_4(k) e^{-\theta_{32}} \\
 0 & 0 & 1 \end{pmatrix}, \nonumber \\ 
& v_{3} = \begin{pmatrix}
1 & 0 & 0 \\
0 & 1-|r_{2}(\omega k)|^{2} & -r_{2}^{*}(\omega k)e^{-\theta_{32}} \\
0 & r_{2}(\omega k)e^{\theta_{32}} & 1
\end{pmatrix}, && v_{4} = \begin{pmatrix}
1 & r_{3}^{*}(\omega^{2}k)e^{-\theta_{21}} & 0 \\
0 & 1 & 0 \\
0 & -r_{4}^{*}(\omega^{2}k)e^{\theta_{32}} & 1
\end{pmatrix}, \nonumber \\
& v_{5} = \begin{pmatrix}
1-|r_{1}^{*}(\omega^{2}k)|^{2} & 0 & r_{1}^{*}(\omega^{2}k)e^{-\theta_{31}} \\
0 & 1 & 0  \\
-r_{1}(\omega^{2}k)e^{\theta_{31}} & 0 & 1
\end{pmatrix}, && v_{6} = \begin{pmatrix}
1 & r_{4}(\omega^{2}k)e^{-\theta_{21}} & 0 \\
0 & 1 & 0 \\
0 & -r_{3}(\omega^{2}k)e^{\theta_{32}} & 1
\end{pmatrix}, \nonumber \\
&  v_7 = \begin{pmatrix}  
  1 - |r_2(k)|^2 & -r_2^*(k) e^{-\theta_{21}} & 0 \\
  r_2(k)e^{\theta_{21}} & 1 & 0 \\
  0 & 0 & 1
   \end{pmatrix}, && v_{8} = \begin{pmatrix}
   1 & 0 & 0 \\
   -r_{4}^{*}(\omega k)e^{\theta_{21}} & 1 & 0 \\
   r_{3}^{*}(\omega k)e^{\theta_{31}} & 0 & 1
\end{pmatrix}, \nonumber \\
& v_{9} = \begin{pmatrix}
1 & 0 & 0 \\
0 & 1 & -r_{1}(\omega k)e^{-\theta_{32}} \\
0 & r_{1}^{*}(\omega k)e^{\theta_{32}} & 1 - |r_{1}(\omega k)|^{2}
\end{pmatrix}, && v_{10} = \begin{pmatrix}
1 & 0 & 0 \\
-r_{3}(\omega k)e^{\theta_{21}} & 1 & 0 \\
r_{4}(\omega k)e^{\theta_{31}} & 0 & 1
\end{pmatrix}, \nonumber \\
& v_{11} = \begin{pmatrix}
1 & 0 & r_{2}(\omega^{2}k)e^{-\theta_{31}}  \\
0 & 1 & 0 \\
-r_{2}^{*}(\omega^{2}k)e^{\theta_{31}} & 0 & 1-|r_{2}(\omega^{2}k)|^{2}
\end{pmatrix}, && v_{12} = \begin{pmatrix}
1 & 0 & -r_{4}^{*}(k)e^{-\theta_{31}} \\
0 & 1 & r_{3}^{*}(k) e^{-\theta_{32}} \\
0 & 0 & 1
\end{pmatrix}, \label{vdef}
\end{align}
where $v_j$ denotes the restriction of $v$  to the subcontour of $\Gamma$ labeled by $j$ in Figure \ref{fig: Dn}. In particular, the following symmetries hold:
\begin{align}\label{vsymm}
v(x,t,k) = \mathcal{A} v(x,t,\omega k)\mathcal{A}^{-1}  = \mathcal{B} \overline{v(x,t, \bar{k})}^{-1}\mathcal{B}, \qquad k \in \Gamma,
\end{align}
where
\begin{align}\label{def of Acal and Bcal}
\mathcal{A} := \begin{pmatrix}
0 & 0 & 1 \\
1 & 0 & 0 \\
0 & 1 & 0
\end{pmatrix} \qquad \mbox{ and } \qquad \mathcal{B} := \begin{pmatrix}
0 & 1 & 0 \\
1 & 0 & 0 \\
0 & 0 & 1
\end{pmatrix}.
\end{align}
Let $D_{1}, D_{2}$ be the two open subsets of the plane given by (see also Figure \ref{fig: Dn})
\begin{align*}
D_{1} & = \{ k \in \C \,|\, \re l_{1} < \re l_{2} < \re l_{3} \mbox{ and } \re z_{2} < \re z_{1} < \re z_{3} \}, \\
D_{2} & = \{ k \in \C \,|\, \re l_{1} < \re l_{2} < \re l_{3} \mbox{ and } \re z_{2} < \re z_{3} < \re z_{1} \}.
\end{align*}

We consider the following RH problem.

\begin{RHproblem}[RH problem for $M$]\label{RH problem for M}
Find a $3 \times 3$-matrix valued function $M(x,t,k)$ with the following properties:
\begin{enumerate}[(a)]
\item $M(x,t,\cdot) : \mathbb{C}\setminus \Gamma \to \mathbb{C}^{3 \times 3}$ is analytic.

\item The limits of $M(x,t,k)$ as $k$ approaches $\Gamma\setminus \{0\}$ from the left and right exist, are continuous on $\Gamma\setminus \{0\}$, and are denoted by $M_+$ and $M_-$, respectively. Furthermore, they are related by
\begin{align}\label{Mjumpcondition}
& M_{+}(x,t,k) = M_{-}(x,t,k)v(x,t,k), \qquad k \in \Gamma.
\end{align}

\item As $k \to \infty$, $k \notin \Gamma$, we have
\begin{align*}
M(x,t,k) = I + \frac{M^{(1)}(x,t)}{k} + \frac{M^{(2)}(x,t)}{k^{2}} + O\bigg(\frac{1}{k^3}\bigg),
\end{align*}
where the matrices $M^{(1)}$ and $M^{(2)}$ depend on $x$ and $t$ but not on $k$, and satisfy
\begin{align}\label{singRHMatinftyb}
M_{12}^{(1)} = M_{13}^{(1)} = M_{12}^{(2)} + M_{13}^{(2)} = 0.
\end{align}

\item There exist matrices $\{\mathcal{M}_j^{(l)}(x,t)\}_{l=-2}^{+\infty}$, $j=1,2$, depending on $x$ and $t$ but not on $k$ such that, for any $N \geq -2$,
\begin{align}\label{singRHMat0}
& M(x,t,k) = \sum_{l=-2}^{N} \mathcal{M}_1^{(l)}(x,t)k^{l} + O(k^{N+1}) \qquad \text{as}\ k \to 0, \ k \in D_1, \\
& M(x,t,k) = \sum_{l=-2}^{N} \mathcal{M}_2^{(l)}(x,t)k^{l} + O(k^{N+1}) \qquad \text{as}\ k \to 0, \ k \in D_2.
\end{align}
Furthermore, there exist scalar coefficients $\alpha, \beta, \gamma, \delta, \epsilon_{1}, \epsilon_{2}$ depending on $x$ and $t$, but not on $k$, such that
\begin{align} \label{explicit Mcalpm2p}
\mathcal{M}_{1}^{(-2)}(x,t) = \mathcal{M}_{2}^{(-2)}(x,t) = &\; \alpha(x,t) \begin{pmatrix}
\omega & 0 & 0 \\
\omega & 0 & 0 \\
\omega & 0 & 0
\end{pmatrix}, 
	\\
\mathcal{M}_{1}^{(-1)}(x,t) = \mathcal{M}_{2}^{(-1)}(x,t)  = &\; \beta(x,t) \begin{pmatrix}
\omega^{2} & 0 & 0 \\
\omega^{2} & 0 & 0 \\
\omega^{2} & 0 & 0 
\end{pmatrix} + \gamma(x,t) \begin{pmatrix}
\omega^{2} & 0 & 0 \\
1 & 0 & 0 \\
\omega & 0 & 0
\end{pmatrix} \nonumber \\
& + \delta(x,t) \begin{pmatrix}
0 & 1-\omega & 0 \\
0 & 1-\omega & 0 \\
0 & 1-\omega & 0
\end{pmatrix}, \label{explicit Mcalpm1p}
\end{align}
and the third columns of $\mathcal{M}_{1}^{(0)}(x,t)$ and $\mathcal{M}_{2}^{(0)}(x,t)$ are given by
\begin{align}
[\mathcal{M}_{1}^{(0)}(x,t)]_{3} = \epsilon_{1}(x,t) \begin{pmatrix}
1 \\
1 \\
1 
\end{pmatrix}, \qquad [\mathcal{M}_{2}^{(0)}(x,t)]_{3} = \epsilon_{2}(x,t) \begin{pmatrix}
1 \\
1 \\
1 
\end{pmatrix}, \label{explicit Mcalp0p third column}
\end{align}
where $[A]_j$ denotes the $j$th column of a matrix $A$.

\item $M$ satisfies the symmetries
\begin{align}\label{singRHsymm}
M(x,t, k) = \mathcal{A} M(x,t,\omega k)\mathcal{A}^{-1} = \mathcal{B} \overline{M(x,t,\overline{k})}\mathcal{B}, \qquad k \in \C \setminus \Gamma.
\end{align}
\end{enumerate}
\end{RHproblem}

The functions $\alpha, \beta, \gamma, \delta, \epsilon_{1}, \epsilon_{2}$ are not prescribed in the formulation of RH problem \ref{RH problem for M}. The solution $M$ of RH problem \ref{RH problem for M} is unique (this can be proved in the same way as in \cite[Appendix A]{CLgoodboussinesq}).
%


\subsubsection{Statement of the second theorem}
Our main theorem states that the solution $\{u(x,t),$ $ v(x,t)\}$ of the Boussinesq equation \eqref{boussinesqsystem} can be recovered from the solution $M(x,t,k)$ of RH problem \ref{RH problem for M} via the relations (\ref{recoveruv}).

\begin{definition}\upshape
Let $T > 0$. We call $\{u(x,t), v(x,t)\}$ a {\it Schwartz class solution of \eqref{boussinesqsystem} on $\R_+ \times [0, T]$ with initial data $u_0, v_0 \in \mathcal{S}(\R_{+})$ and boundary values $\tilde{u}_{0}, \tilde{u}_{1}, \tilde{u}_{2}, \tilde{v}_0 \in C^{\infty}([0,T])$} if
\begin{enumerate}[$(i)$] 
  \item $u,v$ are smooth real-valued functions of $(x,t) \in \R_{+} \times [0,T]$.

\item $u,v$ satisfy \eqref{boussinesqsystem} for $(x,t) \in \R_{+} \times [0,T]$ and 
\begin{align*}
& u(x,0) = u_0(x), \quad v(x,0) = v_0(x), \qquad x \in \R_{+}, \\
& u(0,t) = \tilde{u}_{0}(t), \quad u_{x}(0,t) = \tilde{u}_{1}(t), \quad u_{xx}(0,t) = \tilde{u}_{2}(t), \quad v(0,t) = \tilde{v}_0(t), \qquad t \in [0,T].
\end{align*}
  \item $u,v$ have rapid decay as $x \to +\infty$ in the sense that, for each integer $N \geq 1$,
$$\sup_{\substack{x \in \R_{+} \\ t \in [0, T]}} \sum_{i =0}^N (1+|x|)^N(|\partial_x^i u| + |\partial_x^i v| ) < \infty.$$
\end{enumerate} 
\end{definition}

The second theorem can now be stated.

\begin{theorem}[Solution of (\ref{boussinesqsystem}) via inverse scattering]\label{RHth}
Suppose $\{u(x,t), v(x,t)\}$ is a Schwartz class solution of (\ref{boussinesqsystem}) on $\R_+ \times [0, T]$ for some $T > 0$, with initial data $u_0, v_0 \in \mathcal{S}(\R_{+})$ and boundary values $\tilde{u}_{0}, \tilde{u}_{1}, \tilde{u}_{2}, \tilde{v}_0 \in C^{\infty}([0,T])$ such that Assumptions \ref{solitonlessassumption} and \ref{originassumption} hold. Define the spectral functions $r_j(k)$, $j = 1,2,3,4$, in terms of $u_0, v_0, \tilde{u}_{0}, \tilde{u}_{1}, \tilde{u}_{2}, \tilde{v}_0$ by (\ref{r1r2def}).
Then RH problem \ref{RH problem for M} has a unique solution $M(x,t,k)$ for each $(x,t) \in \R_{+} \times [0,T]$ and the formulas \begin{align}\label{recoveruv}
\begin{cases}
 \displaystyle{u(x,t) = -\frac{3}{2}\frac{\partial}{\partial x}\lim_{k\to \infty}k\big[(M(x,t,k))_{33} - 1\big],}
	\vspace{.1cm}\\
 \displaystyle{v(x,t) = -\frac{3}{2}\frac{\partial}{\partial t}\lim_{k\to \infty}k\big[(M(x,t,k))_{33} - 1\big],}
\end{cases}
\end{align}
which express $\{u(x,t), v(x,t)\}$ in terms of $M$, are valid for all $(x,t) \in \R_{+} \times [0,T]$.
\end{theorem}

The proof of Theorem \ref{RHsec}, which does not rely on Theorem \ref{r1r2th}, is presented in Section \ref{RHsec}.

The RH problem \ref{RH problem for M} is singular at the origin in the sense that the solution $M$ is allowed to have a double pole at $k = 0$. 
It is sometimes convenient to have a RH problem which is regular at the origin.
Such a RH problem can be obtained by introducing the row-vector-valued function $n$ by
\begin{align}\label{ndef}
n(x,t,k) = \begin{pmatrix}\omega & \omega^2 & 1 \end{pmatrix} M(x,t,k).
\end{align}
Indeed, since the coefficients of $k^{-2}$ and $k^{-1}$ in the expansion of $M$ at $k = 0$ are such that they vanish when premultiplied by the row-vector $(\omega,\omega^{2},1)$ (see (\ref{explicit Mcalpm2p}) and (\ref{explicit Mcalpm1p})), the function $n$ satisfies the following vector RH problem.

\begin{RHproblem}[RH problem for $n$]\label{RHn}
Find a $1 \times 3$-row-vector valued function $n(x,t,k)$ with the following properties:
\begin{enumerate}[(a)]
\item $n(x,t,\cdot) : \C \setminus \Gamma \to \mathbb{C}^{1 \times 3}$ is analytic.

\item The limits of $n(x,t,k)$ as $k$ approaches $\Gamma \setminus \{0\}$ from the left and right exist, are continuous on $\Gamma \setminus \{0\}$, and are denoted by $n_+$ and $n_-$, respectively. Furthermore, they are related by
\begin{align}\label{njump}
  n_+(x,t,k) = n_-(x, t, k) v(x, t, k), \qquad k \in \Gamma \setminus \{0\}.
\end{align}

\item $n(x,t,k) = (\omega,\omega^{2},1) + O(k^{-1})$ as $k \to \infty$.

\item $n(x,t,k) = O(1)$ as $k \to 0$.
\end{enumerate}
\end{RHproblem}
The RH problem for $n$ is regular at the origin and is clearly simpler than the RH problem for $M$. Moreover, the solution $\{u, v\}$ of (\ref{boussinesqsystem}) can be recovered from $n$ via the relations
\begin{align}\label{recoveruvn}
\begin{cases}
u(x,t) = -\frac{3}{2}\frac{\partial}{\partial x}\lim_{k\to \infty}k(n_{3}(x,t,k) - 1),
	\\
v(x,t) = -\frac{3}{2}\frac{\partial}{\partial t}\lim_{k\to \infty}k(n_{3}(x,t,k) - 1).
\end{cases}
\end{align}
The following corollary is a consequence of Theorem \ref{RHth}.

\begin{corollary}[Solution of (\ref{boussinesqsystem}) in terms of $n$]\label{ncor}
Suppose the assumptions of Theorem \ref{RHth} hold. Let $U$ be an open subset of $\R \times [0,T]$ and suppose that the solution of RH problem \ref{RHn} for $n$ is unique for each $(x,t) \in U$ whenever it exists. Then the unique solution $n(x,t,k)$ of RH problem \ref{RHn} exists for each $(x,t) \in U$ and the formulas (\ref{recoveruvn}) are valid for all $(x,t) \in U$.
\end{corollary}

\begin{remark}
We have not been able to establish uniqueness of the solution of the RH problem for $n$ except in special cases (this is the reason why we have chosen to formulate Theorem \ref{RHth} in terms of $M$ rather than $n$). 
\end{remark}

\section{Spectral analysis}\nequation\label{specsec}
\subsection{Preliminaries}

The system \eqref{boussinesqsystem} is the compatibility condition of the Lax pair
\begin{equation}\label{Xlax}
\begin{cases}
\mu_{x} - [\mathcal{L},\mu] = \mathsf{U} \mu, \\
\mu_{t} - [\mathcal{Z},\mu] = \mathsf{V} \mu,
\end{cases}
\end{equation}
where $\mathcal{L} := \mbox{diag}(l_{1},l_{2},l_{3})$, $\mathcal{Z} := \mbox{diag}(z_{1},z_{2},z_{3})$,
\begin{align}
& \mathsf{U} := L-\mathcal{L}, \qquad \mathsf{V}:=Z-\mathcal{Z}, \qquad L:=P^{-1}\tilde{L}P, \qquad Z:=P^{-1}\tilde{Z}P, \label{U and V def} \\
& \tilde{L} := \begin{pmatrix}
0 & 1 & 0 \\
0 & 0 & 1 \\
k^{3}-u_{x}-v & -2u & 0
\end{pmatrix}, \qquad \tilde{Z} := \begin{pmatrix}
\frac{4}{3}u & 0 & 1 \\
k^{3} + \frac{u_{x}}{3}-v & -\frac{2}{3}u & 0 \\
\frac{u_{xx}}{3}-v_{x} & k^{3} - \frac{u_{x}}{3}-v & - \frac{2}{3}u
\end{pmatrix}, \\
& P(k) := \begin{pmatrix}
\omega & \omega^{2} & 1  \\
\omega^{2} k & \omega k & k \\
k^{2} & k^{2} & k^{2}
\end{pmatrix},
\end{align}
and $\mu(x,t,k)$ is a $3\times 3$-matrix valued eigenfunction depending on $x$ and $t$ as well as the spectral parameter $k \in \C$.
The adjoint system of \eqref{Xlax} is
\begin{equation}\label{Xlax adjoint}
\begin{cases}
\mu_{x}^{A} + [\mathcal{L},\mu^{A}] = -\mathsf{U}^{T} \mu^{A}, \\
\mu_{t}^{A} + [\mathcal{Z},\mu^{A}] = -\mathsf{V}^{T} \mu^{A},
\end{cases}
\end{equation}
in the sense that if $\mu(x,t,k)$ is a $3\times 3$ invertible matrix-valued function satisfying \eqref{Xlax}, then $\mu^{A}:=(\mu^{-1})^{T}$ satisfies \eqref{Xlax adjoint}. 

%

We write the systems (\ref{Xlax}) and \eqref{Xlax adjoint} in differential forms as
\begin{equation}\label{mulaxdiffform}
d\left(e^{-\hat{\mathcal{L}}x - \hat{\mathcal{Z}}t} \mu \right) = W, \qquad d\left(e^{\hat{\mathcal{L}}x + \hat{\mathcal{Z}}t} \mu^{A} \right) = W^{A},
\end{equation}
where $W(x,t,k)$ and $W^{A}(x,t,k)$ are the closed one-forms defined by
\begin{equation}\label{Wdef}  
W = e^{-\hat{\mathcal{L}}x - \hat{\mathcal{Z}}t}[(\mathsf{U} dx + \mathsf{V} dt) \mu], \qquad W^{A} = -e^{\hat{\mathcal{L}}x + \hat{\mathcal{Z}}t}[(\mathsf{U}^{T} dx + \mathsf{V}^{T} dt) \mu^{A}].
\end{equation}  
It is also convenient for us to introduce the following twelve open, pairwisely disjoint subsets $\{D_n\}_1^{12}$ of the plane (see Figure \ref{fig: Dn}):
\begin{align*}
D_{1} & = \{ k \in \C \,|\, \re l_{1} < \re l_{2} < \re l_{3} \mbox{ and } \re z_{2} < \re z_{1} < \re z_{3} \}, \\
D_{2} & = \{ k \in \C \,|\, \re l_{1} < \re l_{2} < \re l_{3} \mbox{ and } \re z_{2} < \re z_{3} < \re z_{1} \}, \\
D_{3} & = \{ k \in \C \,|\, \re l_{1} < \re l_{3} < \re l_{2} \mbox{ and } \re z_{3} < \re z_{2} < \re z_{1} \}, \\
D_{4} & = \{ k \in \C \,|\, \re l_{1} < \re l_{3} < \re l_{2} \mbox{ and } \re z_{3} < \re z_{1} < \re z_{2} \}, \\
D_{5} & = \{ k \in \C \,|\, \re l_{3} < \re l_{1} < \re l_{2} \mbox{ and } \re z_{1} < \re z_{3} < \re z_{2} \}, \\
D_{6} & = \{ k \in \C \,|\, \re l_{3} < \re l_{1} < \re l_{2} \mbox{ and } \re z_{1} < \re z_{2} < \re z_{3} \}, \\
D_{7} & = \{ k \in \C \,|\, \re l_{3} < \re l_{2} < \re l_{1} \mbox{ and } \re z_{2} < \re z_{1} < \re z_{3} \}, \\
D_{8} & = \{ k \in \C \,|\, \re l_{3} < \re l_{2} < \re l_{1} \mbox{ and } \re z_{2} < \re z_{3} < \re z_{1} \}, \\
D_{9} & = \{ k \in \C \,|\, \re l_{2} < \re l_{3} < \re l_{1} \mbox{ and } \re z_{3} < \re z_{2} < \re z_{1} \}, \\
D_{10} & = \{ k \in \C \,|\, \re l_{2} < \re l_{3} < \re l_{1} \mbox{ and } \re z_{3} < \re z_{1} < \re z_{2} \}, \\
D_{11} & = \{ k \in \C \,|\, \re l_{2} < \re l_{1} < \re l_{3} \mbox{ and } \re z_{1} < \re z_{3} < \re z_{2} \}, \\
D_{12} & = \{ k \in \C \,|\, \re l_{2} < \re l_{1} < \re l_{3} \mbox{ and } \re z_{1} < \re z_{2} < \re z_{3} \}.
\end{align*}
\subsection{The eigenfunctions $\mu_{1}, \mu_{2}, \mu_{3}$, and $\mu_{1}^{A}, \mu_{2}^{A}, \mu_{3}^{A}$}
We consider three contours $\{\gamma_j\}_1^3$ in the region \eqref{halflinedomain} going from $(x_j, t_j)$ to $(x,t)$ as indicated in Figure \ref{mucontours.pdf}, where $(x_1, t_1) = (0, T)$, $(x_2, t_2) = (0, 0)$, and $(x_3, t_3) = (+\infty, t)$. This choice implies the following inequalities for $(x',t')$ on the contours,
\begin{align}
\gamma_1: x' - x \leq 0,& \qquad t' - t \geq 0, \nonumber \\
\gamma_2: x' - x \leq 0,& \qquad t' - t \leq 0, \label{contourinequalities}
	\\ 
\gamma_3: x' - x \geq 0.& \nonumber
\end{align}

\begin{figure}
\begin{center}
\begin{tikzpicture}[master]
\draw[line width=0.15 mm, dashed,fill=gray!25] (0,0)--(3.5,0)--(3.5,1.2)--(0,1.2)--(0,0) -- cycle;
\draw[line width=0.2 mm,-<-=0,->-=1] (0,1.8)--(0,0)--(3.5,0);
\draw[line width=0.25 mm, white] (3.5,0)--(3.5,1.8);
\draw[fill] (1.35,0.6) circle (0.04); 
\draw[fill] (0,1.2) circle (0.04); 
\draw[line width=0.35 mm,->-=0.23,->-=0.75] (0,1.2)--(0,0.6)--(1.35,0.6);
\node at (1.85,0.6) {\small $(x,t)$};
\node at (-0.25,1.2) {\small $T$};
\node at (1.75,-0.5) {$\gamma_{1}$};
\end{tikzpicture}
 \hspace{0.6cm}
\begin{tikzpicture}[slave]
\draw[line width=0.15 mm, dashed,fill=gray!25] (0,0)--(3.5,0)--(3.5,1.2)--(0,1.2)--(0,0) -- cycle;
\draw[line width=0.2 mm,-<-=0,->-=1] (0,1.8)--(0,0)--(3.5,0);
\draw[line width=0.25 mm, white] (3.5,0)--(3.5,1.8);
\draw[fill] (1.35,0.6) circle (0.04); 
\draw[fill] (0,0) circle (0.04); 
\draw[line width=0.35 mm,->-=0.23,->-=0.75] (0,0)--(0,0.6)--(1.35,0.6);
\node at (1.85,0.6) {\small $(x,t)$};
\node at (-0.25,1.2) {\small $T$};
\node at (1.75,-0.5) {$\gamma_{2}$};
\end{tikzpicture} \hspace{0.6cm}
\begin{tikzpicture}[slave]
\draw[line width=0.15 mm, dashed,fill=gray!25] (0,0)--(3.5,0)--(3.5,1.2)--(0,1.2)--(0,0) -- cycle;
\draw[line width=0.2 mm,-<-=0,->-=1] (0,1.8)--(0,0)--(3.5,0);
\draw[line width=0.25 mm, white] (3.5,0)--(3.5,1.8);
\draw[fill] (1.35,0.6) circle (0.04); 
\draw[fill] (3.5,0.6) circle (0.04); 
\draw[line width=0.35 mm,->-=0.6] (3.5,0.6)--(1.35,0.6);
\node at (0.85,0.6) {\small $(x,t)$};
\node at (-0.25,1.2) {\small $T$};
\node at (1.75,-0.5) {$\gamma_{3}$};
\end{tikzpicture}
\begin{figuretext}\label{mucontours.pdf}
       The contours $\gamma_1$, $\gamma_2$, and $\gamma_3$ in the $(x, t)$-plane.
   \end{figuretext}
\end{center}
\end{figure}

We define three eigenfunctions $\{\mu_j\}_1^3$ of (\ref{Xlax}) and three eigenfunctions $\{\mu^{A}_j\}_1^3$ of (\ref{Xlax adjoint}) as the solutions to the following Volterra integral equations:
\begin{align}
& \mu_j(x,t,k) = I +  \int_{\gamma_j} e^{\hat{\mathcal{L}}(k)x + \hat{\mathcal{Z}}(k)t} W_j(x',t',k), & & j = 1, 2,3, \label{mujdef} \\
& \mu^{A}_j(x,t,k) = I +  \int_{\gamma_j} e^{-\hat{\mathcal{L}}(k)x - \hat{\mathcal{Z}}(k)t} W_j^{A}(x',t',k), & & j = 1, 2,3, \label{XAjdef}
\end{align}
where $W_j$ is given by the first right-hand side in (\ref{Wdef}) with $\mu$ replaced with $\mu_j$, and similarly $W_j^{A}$ is given by the second right-hand side in \eqref{Wdef} with $\mu^{A}$ replaced with $\mu^{A}_j$. 


The third column of the matrix equation (\ref{mujdef}) involves the exponentials
$$e^{(l_1 - l_3)(x - x') + (z_1 - z_3)(t - t')}, \qquad e^{(l_2 - l_3)(x - x') + (z_2 - z_3)(t - t')}.$$ 
Using the inequalities in (\ref{contourinequalities}) it follows that these exponentials are bounded in the following regions of the complex $k$-plane:
\begin{align*}
\gamma_1: \{\re  l_1 \leq \re  l_3\} \cap \{\re  l_2 \leq \re  l_3\} &\cap \{\re  z_3 \leq \re  z_1\} \cap \{\re  z_3 \leq \re  z_2\}, \\
\gamma_2: \{\re  l_1 \leq \re  l_3\} \cap \{\re  l_2 \leq \re  l_3\} &\cap \{\re  z_1 \leq \re  z_3\} \cap \{\re  z_2 \leq \re  z_3\}, \\
\gamma_3: \{\re  l_3 \leq \re  l_1\} \cap \{\re  l_3 \leq \re  l_2\}&.
\end{align*}
These boundedness properties, combined with a standard analysis of the Volterra integral equations (\ref{mujdef}), imply that the third column vectors of $\mu_2$ and $\mu_3$ are bounded (as functions of $x\in \R_{+}$ and $t\in [0,T]$) for each fixed $k \in \mathcal{T}$ and $k \in \mathcal{S}$, respectively,  where the sets $\mathcal{T}$ and $\mathcal{S}$ are defined by (see Figure \ref{fig: Dn})
\begin{align*}
& \mathcal{T} = \{ k \in \C \, | \, \arg k \in (-\tfrac{\pi}{6},\tfrac{\pi}{6})\} = \mbox{int} (\bar{D}_{1}\cup\bar{D}_{12}), \\
& \mathcal{S} = \{k \in \C \, | \, \arg k \in (\tfrac{2\pi}{3},\pi]\cup (-\pi,-\tfrac{2\pi}{3})\} = \mbox{int} (\bar{D}_{5}\cup\bar{D}_{6}\cup\bar{D}_{7}\cup\bar{D}_{8}).
\end{align*}
The analogous region of boundedness for the third column of $\mu_1$ is in general empty.
Similar conditions are valid for the other column vectors. 

Let $\mathcal{R} = \mbox{int}(\bar{D}_3 \cup \bar{D}_4)$ and, for a set $U \subset \C$, $\check{U} := U \cup (-U)$.

 We summarize the properties of $\mu_{1},\mu_{2},\mu_{3}$ is the following proposition.

\begin{proposition}[Basic properties of $\mu_{1},\mu_{2},\mu_{3}$]\label{XYprop}
Suppose $u_0, v_0 \in \mathcal{S}(\R_{+})$ and $\tilde{u}_{0}, \tilde{u}_{1}, \tilde{u}_{2}, \tilde{v}_0 \in C^{\infty}([0,T])$. 
Then the equations (\ref{mujdef}) uniquely define three $3 \times 3$-matrix valued solutions $\mu_{1},\mu_{2},\mu_{3}$ of the $x$-part of (\ref{Xlax}) with the following properties:
\begin{enumerate}[$(a)$]
\item The functions $\mu_{j}$ have the following domains of definition:
\begin{subequations}\label{lol4main}
\begin{align}
& \text{$\mu_1(x,t,k)$ and $\mu_2(x,t,k)$ are defined for $x\in \R_{+}$, $t \in [0,T]$, $k\in \C\setminus\{0\}$,} \\
& \mu_3(x,t,k)\; \text{is defined for} \; x\in \R_{+}, \; t \in [0,T], \; k \in (\omega^{2} \overline{\mathcal{S}}, \omega \overline{\mathcal{S}}, \overline{\mathcal{S}})\setminus\{0\},
\end{align}
\end{subequations}
where the notation $k \in (\omega^{2} \overline{\mathcal{S}}, \omega \overline{\mathcal{S}}, \overline{\mathcal{S}})\setminus\{0\}$ indicates that the first, second, and third columns of $\mu_3(x,t,k)$ are well-defined for $k$ in the sets $\omega^2\overline{\mathcal{S}}\setminus\{0\}$, $\omega \overline{\mathcal{S}}\setminus\{0\}$, and $\overline{\mathcal{S}}\setminus\{0\}$, respectively.

\smallskip \noindent For each $k$ as in \eqref{lol4main}, $\mu_{j}(\cdot,\cdot, k)$ is smooth and satisfies the $x$-part of (\ref{Xlax}).

\item \vspace{0.1cm} For each $x \in \R_{+}$ and $t\in [0,T]$, 
\begin{subequations}
\begin{align*}
& \text{$\mu_1(x,t,k)$ and $\mu_2(x,t,k)$ are analytic for $k \in \C\setminus\{0\}$}, \\
& \mu_3(x,t,k) \text{ is continuous for } k \in (\omega^{2} \overline{\mathcal{S}}, \omega \overline{\mathcal{S}}, \overline{\mathcal{S}})\setminus\{0\} \mbox{ and analytic for } k \in (\omega^{2} \mathcal{S}, \omega \mathcal{S}, \mathcal{S})\setminus\{0\}.
\end{align*}
\end{subequations}
\item For each $x \in \R_{+}$, $t\in [0,T]$, and $j = 0,1,\dots$,
\begin{subequations}
\begin{align*}
& \tfrac{\partial^j }{\partial k^j}\mu_3(x,t,\cdot) \text{ has a continuous extension to } k \in (\omega^{2} \overline{\mathcal{S}}, \omega \overline{\mathcal{S}}, \overline{\mathcal{S}})\setminus\{0\}.
\end{align*}
\end{subequations}

\item For each $n \geq 1$ and $\epsilon > 0$, there are bounded smooth positive functions $\{f_{j}(x,t)\}_{j=2}^{3}$ and $\{g_{j}(t)\}_{j=1}^{2}$ of $x \in \R_{+}$ and $t\in [0,T]$, such that the following estimates hold for $x \in \R_{+}$ and $t\in [0,T]$ and $ j = 0, 1, \dots, n$:
\begin{subequations}\label{region of boundedness of muj}
\begin{align}
& \big|\tfrac{\partial^j}{\partial k^j}\mu_{2}(x,t,k) \big| \leq
f_{2}(x,t), & & k \in (\omega^{2} \overline{\mathcal{T}}, \omega \overline{\mathcal{T}}, \overline{\mathcal{T}}), \ |k| > \epsilon,  \\
& \big|\tfrac{\partial^j}{\partial k^j}\big(\mu_{3}(x,t,k) - I\big) \big| \leq
f_3(x,t), & & k \in (\omega^{2} \overline{\mathcal{S}}, \omega \overline{\mathcal{S}}, \overline{\mathcal{S}}), \ |k| > \epsilon, \\
& \big|\tfrac{\partial^j}{\partial k^j}\mu_{1}(0,t,k) \big| \leq
g_1(t), & & k \in (\omega^{2}\overline{\check{\mathcal{R}}}, \omega\overline{\check{\mathcal{R}}}, \overline{\check{\mathcal{R}}} ), \ |k| > \epsilon,  	\\ 
& \big|\tfrac{\partial^j}{\partial k^j}\mu_{2}(0,t,k) \big| \leq
g_{2}(t), & & k \in (\omega^{2}\overline{\check{\mathcal{T}}}, \omega \overline{\check{\mathcal{T}}}, 
\overline{\check{\mathcal{T}}}), \ |k| > \epsilon,  
\end{align}
\end{subequations}
Moreover, $f_{3}(x,t)$ has rapid decay as $x \to +\infty$.

\item $\mu_{1},\mu_{2},\mu_{3}$ obey the following symmetries for each $x \in \R_{+}$ and $t\in [0,T]$:
\begin{align*}
&  \mu_{j}(x,t, k) = \mathcal{A} \mu_{j}(x,t,\omega k)\mathcal{A}^{-1} = \mathcal{B} \overline{\mu_{j}(x,t,\overline{k})}\mathcal{B} \quad \text{for} \; k \in \C \setminus\{0\}, \; j=1,2, \\
&  \mu_{3}(x,t, k) = \mathcal{A} \mu_{3}(x,t,\omega k)\mathcal{A}^{-1} = \mathcal{B} \overline{\mu_{3}(x,t,\overline{k})}\mathcal{B}\quad \text{for} \;  k \in (\omega^{2} \overline{\mathcal{S}}, \omega \overline{\mathcal{S}}, \overline{\mathcal{S}})\setminus\{0\}.
\end{align*}

\item For each  $x \in \R_{+}$, $t\in [0,T]$, and $j=1,2$, $\det \mu_{j}(x,t,k)=1$.

\smallskip \noindent If $u_0, v_0$ have compact support, then, for each  $x \in \R_{+}$ and $t\in [0,T]$,  $\mu_{3}(x,t, k)$ is defined and analytic for $k \in \C \setminus \{0\}$ and $\det \mu_{3}(x,t,k)=1$.

\end{enumerate}
\end{proposition}

We now turn to the properties of $\mu_{1}^{A},\mu_{2}^{A},\mu_{3}^{A}$.
\begin{proposition}[Basic properties of $\mu_{1}^{A},\mu_{2}^{A},\mu_{3}^{A}$]\label{XAYAprop}
Suppose $u_0, v_0 \in \mathcal{S}(\R_{+})$ and $\tilde{u}_{0}, \tilde{u}_{1}, \tilde{u}_{2}, \tilde{v}_0 \in C^{\infty}([0,T])$. 
Then the equations (\ref{XAjdef}) uniquely define three $3 \times 3$-matrix valued solutions $\mu_{1}^{A},\mu_{2}^{A},\mu_{3}^{A}$ of the $x$-part of (\ref{Xlax adjoint}) with the following properties:
\begin{enumerate}[$(a)$]
\item The functions $\mu_{j}^{A}$ have the following domains of definition:
\begin{subequations}\label{lol4main adjoint}
\begin{align}\label{lol4 adjoint}
& \text{$\mu_1^A(x,t,k)$ and $\mu_2^A(x,t,k)$ are defined for $x\in \R_{+}$, $t \in [0,T]$, $k\in \C\setminus\{0\}$,} 
	 \\
& \mu_3^{A}(x,t,k)\; \text{is defined for} \; x\in \R_{+}, \; t \in [0,T], \; k \in (-\omega^{2} \overline{\mathcal{S}}, -\omega \overline{\mathcal{S}}, -\overline{\mathcal{S}})\setminus\{0\}.
\end{align}
\end{subequations}

\smallskip \noindent For each $k$ as in \eqref{lol4main adjoint}, $\mu_{j}^{A}(\cdot,\cdot, k)$ is smooth and satisfies the $x$-part of (\ref{Xlax adjoint}).

\item \vspace{0.1cm} For each $x \in \R_{+}$ and $t\in [0,T]$, 
\begin{subequations}
\begin{align*}
& \text{$\mu_1^{A}(x,t,k)$ and $\mu_2^{A}(x,t,k)$ are analytic for $k \in \C\setminus\{0\}$}, 
	\\
& \mu_3^{A}(x,t,k) \text{ is continuous for } k \in (-\omega^{2} \overline{\mathcal{S}}, -\omega \overline{\mathcal{S}}, -\overline{\mathcal{S}})\setminus\{0\} \\
& \hspace{1.69cm} \mbox{ and analytic for } k \in (-\omega^{2} \mathcal{S}, -\omega \mathcal{S}, -\mathcal{S})\setminus\{0\}.
\end{align*}
\end{subequations}
\item For each $x \in \R_{+}$ and $t\in [0,T]$,
\begin{subequations}
\begin{align*}
& \tfrac{\partial^j }{\partial k^j}\mu_3^{A}(x,t,\cdot) \text{ has a continuous extension to } k \in (-\omega^{2} \overline{\mathcal{S}}, -\omega \overline{\mathcal{S}}, -\overline{\mathcal{S}})\setminus\{0\}.
\end{align*}
\end{subequations}

\item For each $n \geq 1$ and $\epsilon > 0$, there are bounded smooth positive functions $\{\tilde{f}_{j}(x,t)\}_{j=1,3}$ and $\{\tilde{g}_{j}(t)\}_{j=1}^{2}$ of $x \in \R_{+}$ and $t\in [0,T]$, such that the following estimates hold for $x \in \R_{+}$ and $t\in [0,T]$ and $ j = 0, 1, \dots, n$:
\begin{subequations}\label{region of boundedness of mujA}
\begin{align}
& \big|\tfrac{\partial^j}{\partial k^j}\mu_{1}^{A}(x,t,k) \big| \leq
\tilde{f}_{1}(x,t), & & k \in (-\omega^{2} \overline{\mathcal{T}}, -\omega \overline{\mathcal{T}}, -\overline{\mathcal{T}}), \ |k| > \epsilon,  \\
& \big|\tfrac{\partial^j}{\partial k^j}\big(\mu_{3}^{A} (x,t,k) - I\big) \big| \leq
\tilde{f}_3(x,t), & & k \in (-\omega^{2} \overline{\mathcal{S}}, -\omega \overline{\mathcal{S}}, -\overline{\mathcal{S}}), \ |k| > \epsilon, \\
& \big|\tfrac{\partial^j}{\partial k^j}\mu_{1}^{A} (0,t,k) \big| \leq
\tilde{g}_1(t), & & k \in (\omega^{2}\overline{\check{\mathcal{T}}}, \omega \overline{\check{\mathcal{T}}}, 
\overline{\check{\mathcal{T}}}), \ |k| > \epsilon,  	\\ 
& \big|\tfrac{\partial^j}{\partial k^j}\mu_{2}^{A}(0,t,k) \big| \leq
\tilde{g}_{2}(t), & & k \in (\omega^{2}\overline{\check{\mathcal{R}}}, \omega\overline{\check{\mathcal{R}}}, \overline{\check{\mathcal{R}}} ), \ |k| > \epsilon,  
\end{align}
\end{subequations}
Moreover, $\tilde{f}_{3}(x,t)$ has rapid decay as $x \to +\infty$. 

\item $\mu_{1}^{A},\mu_{2}^{A},\mu_{3}^{A}$ obey the following symmetries for each $x \in \R_{+}$ and $t\in [0,T]$:
\begin{subequations}
\begin{align*}
&  \mu_{j}^{A}(x,t, k) = \mathcal{A} \mu_{j}^{A}(x,t,\omega k)\mathcal{A}^{-1} = \mathcal{B} \overline{\mu_{j}^{A}(x,t,\overline{k})}\mathcal{B}, \qquad k \in \C \setminus\{0\}, \; j=1,2, \\
&  \mu_{3}^{A}(x,t, k) = \mathcal{A} \mu_{3}^{A}(x,t,\omega k)\mathcal{A}^{-1} = \mathcal{B} \overline{\mu_{3}^{A}(x,t,\overline{k})}\mathcal{B}, \qquad k \in (-\omega^{2} \overline{\mathcal{S}}, -\omega \overline{\mathcal{S}}, -\overline{\mathcal{S}})\setminus\{0\}.
\end{align*}
\end{subequations}

\item $\mu_{1}^A = (\mu_1^{-1})^T$ and $\mu_{2}^A = (\mu_2^{-1})^T$ for $k \in \C \setminus \{0\}$.
If $u_0, v_0$ have compact support, then $\mu_{3}^A = (\mu_3^{-1})^T$ for $k \in \C \setminus \{0\}$.

\end{enumerate}
\end{proposition}

A similar analysis as in \cite[Proposition 3.2]{CLgoodboussinesq} shows that, as $k\to \infty$ within the region of boundedness,
\begin{align}\label{mujlargekexpansion}
\mu_{j}(x,t,k) = I + \frac{\mu_{j}^{(1)}(x,t)}{k} + \frac{\mu_{j}^{(2)}(x,t)}{k^2} + \cdots.
\end{align}
for some matrices $\{\mu_{j}^{(l)}(x,t)\}_{l=1}^{+\infty}$, $j=1,2,3$. 
Let us write
$$ \mathsf{U} = \frac{ \mathsf{U}^{(1)}}{k} + \frac{ \mathsf{U}^{(2)}}{k^2}, \qquad \mathsf{V} = \mathsf{V}^{(0)}+ \frac{ \mathsf{V}^{(1)}}{k} + \frac{ \mathsf{V}^{(2)}}{k^2},$$
where $\mathsf{U}^{(1)}, \mathsf{U}^{(2)}, \mathsf{V}^{(0)}, \mathsf{V}^{(1)}, \mathsf{V}^{(2)}$ are independent of $k$.
Substituting (\ref{mujlargekexpansion}) into (\ref{Xlax}) and identifying terms of the same order yields the relations
\begin{align}\label{xrecursive}
\begin{cases}
[J, \mu_{j}^{(l)}] = \big(\partial_x \mu_{j}^{(l-1)} - \mathsf{U}^{(1)}\mu_{j}^{(l-2)} - \mathsf{U}^{(2)}\mu_{j}^{(l-3)}\big)^{(o)},
	\\
(\partial_x \mu_{j}^{(l)})^{(d)} = \big(\mathsf{U}^{(1)}\mu_{j}^{(l-1)} + \mathsf{U}^{(2)}\mu_{j}^{(l-2)} \big)^{(d)}, \\
(\partial_t \mu_{j}^{(l)})^{(d)} = \big(\mathsf{V}^{(0)}\mu_{j}^{(l)} + \mathsf{V}^{(1)}\mu_{j}^{(l-1)} + \mathsf{V}^{(2)}\mu_{j}^{(l-2)} \big)^{(d)},
\end{cases}
\end{align}
where $J=\mathrm{diag}(\omega,\omega^{2},1)$, and where $A^{(d)}$ and $A^{(o)}$ denote the diagonal and off-diagonal parts of a $3 \times 3$ matrix $A$, respectively.
The coefficients $\{\mu_j^{(l)}(x,t)\}_{j=1,2,3}^{l\geq 1}$ are uniquely determined recursively from (\ref{xrecursive}), the normalizations $\mu_j^{(l)}(x_{j},t_{j}) = 0$, and the initial assignments
\begin{align*}
\mu_{j}^{(-2)} = 0, \qquad \mu_{j}^{(-1)} = 0, \qquad \mu_{j}^{(0)} = I, \qquad j=1,2,3.
\end{align*}
The first two coefficients are given in the following proposition.

\begin{proposition}\label{prop:first two coeff at infty}
The coefficients $\{\mu_{j}^{(l)}\}_{j=1,2,3}^{l=1,2}$ are given by 
\begin{align*}
& \mu_{j}^{(1)}(x,t) =  \int_{(x_{j},t_{j})}^{(x,t)} \Delta^{(1)}(x',t') \begin{pmatrix} \omega^2 & 0 & 0 \\ 
0 & \omega & 0 \\ 
0 & 0 & 1
\end{pmatrix}, \\
& \mu_{j}^{(2)}(x,t) = \frac{2u(x,t)}{3(1-\omega)}\begin{pmatrix}
0 & 1 & -1 \\ - \omega & 0 & \omega \\ \omega^{2} & -\omega^{2} & 0
\end{pmatrix} + \int_{(x_{j},t_{j})}^{(x,t)} \Delta_{j}^{(2)}(x',t') \begin{pmatrix} \omega & 0 & 0 \\ 
0 & \omega^{2} & 0 \\ 
0 & 0 & 1
\end{pmatrix},
\end{align*}
where $\Delta^{(1)}(x,t) = -\frac{2}{3}\big(u(x,t)dx+v(x,t)dt\big)$ and 
\begin{align*}
\Delta_{j}^{(2)}(x,t) = - \frac{u_{x}+v+2u(\mu_{j}^{(1)})_{33}}{3}dx + \frac{4u^{2}+u_{xx}-3v_{x}-6v(\mu_{j}^{(1)})_{33}}{9}dt.
\end{align*}
\end{proposition}

This leads to the following formulas for recovering $u$ and $v$: for $j=1,2,3$,
\begin{align}\label{recoveru}
u(x,t) = -\frac{3}{2}\frac{d}{dx}\lim_{k\to \infty}k(\mu_{j}(x,t,k)-1)_{33}, \\
v(x,t) = -\frac{3}{2}\frac{d}{dt}\lim_{k\to \infty}k(\mu_{j}(x,t,k)-1)_{33},
\end{align}
where the limits must be taken within the corresponding region of boundedness of $[\mu_{j}]_{3}$.

We now turn to the behavior of $\{\mu_{j}(x,t,k)\}_{j=1}^{3}$ as $k\to 0$.  
By transforming the Lax pair equations to a form which is regular at $k=0$, one obtains the following result (see \cite[Proposition 3.3]{CLgoodboussinesq} for a proof in a similar situation).

\begin{proposition}[Asymptotics of $\mu_{j}$ as $k \to 0$]\label{XYat1prop}
Suppose $u_0, v_0 \in \mathcal{S}(\R_{+})$ and $\tilde{u}_{0}, \tilde{u}_{1}, \tilde{u}_{2}, \tilde{v}_0 \in C^{\infty}([0,T])$ and let $p \geq 0$ be an integer. 
Then there are $3 \times 3$-matrix valued functions $C_j^{(l)}(x,t)$, $j = 1, 2, 3$, $l = -2,-1, \dots, p$, with the following properties:
\begin{itemize}
\item For $x \geq 0$, $t\in [0,T]$, and $k$ within the domains of definition \eqref{lol4main} of $\mu_{j}$, the function $\mu_{j}$ satisfies
\begin{align}\label{Xat0}
& \bigg|\frac{\partial^m}{\partial k^m}\big(\mu_{j} - I - \sum_{l=-2}^p C_j^{(l)}(x,t)k^l\big) \bigg| \leq
f_j(x,t)|k|^{p+1-m}, \qquad |k| \leq \frac{1}{2},
\end{align}
where $m \geq 0$ is any integer, $f_j(x,t)$ are smooth positive functions of $x \geq 0$ and $0 \leq t \leq T$, and $f_{3}(x,t)$ has rapid decay as $x \to +\infty$.

\item For each $l \geq -2$ and $j\in \{1,2,3\}$, $C_j^{(l)}(x,t)$ is a smooth function of $x \geq 0$ and $0 \leq t \leq T$, and $C_3^{(l)}(x,t)$ has rapid decay as $x \to +\infty$. 

\item The leading coefficients have the form
\begin{align}
 C_{j}^{(-2)}(x,t)
= &\;
\alpha_{j}(x,t)\begin{pmatrix}
\omega &  \omega^{2} & 1  \\
\omega &  \omega^{2} & 1  \\
\omega &  \omega^{2} & 1  
\end{pmatrix}, \label{Cjminus2} \\
 C_{j}^{(-1)}(x,t) = &\; \beta_{j}(x,t) \begin{pmatrix}
\omega^{2} & \omega & 1 \\
\omega^{2} & \omega & 1 \\
\omega^{2} & \omega & 1 
\end{pmatrix} + \gamma_{j}(x,t) \begin{pmatrix}
\omega^{2} & 1 & \omega \\
1 & \omega & \omega^{2} \\
\omega & \omega^{2} & 1 
\end{pmatrix}, \label{Cjminus1}
	\\
C_{j}^{(0)}(x,t) = & -I + \delta_{j,1}(x,t)\begin{pmatrix}
1 & 1 & 1 \\
1 & 1 & 1 \\
1 & 1 & 1
\end{pmatrix} \nonumber \\
& + \delta_{j,2}(x,t) \begin{pmatrix}
1 & \omega^{2} & \omega \\
\omega & 1 & \omega^{2} \\
\omega^{2} & \omega & 1
\end{pmatrix} + \delta_{j,3}(x,t) \begin{pmatrix}
1 & \omega & \omega^{2} \\
\omega^{2} & 1 & \omega \\
\omega & \omega^{2} & 1
\end{pmatrix}, \label{Cjminus0} \\
C_{j}^{(1)}(x,t) = & \; \epsilon_{j,1}(x,t) \begin{pmatrix}
\omega & \omega^{2} & 1 \\
\omega & \omega^{2} & 1 \\
\omega & \omega^{2} & 1
\end{pmatrix} + \epsilon_{j,2}(x,t) \begin{pmatrix}
\omega & \omega & \omega \\
\omega^{2} & \omega^{2} & \omega^{2} \\
1 & 1 & 1
\end{pmatrix} \nonumber \\
& \; + \epsilon_{j,3}(x,t) \begin{pmatrix}
\omega & 1 & \omega^{2} \\
1 & \omega^{2} & \omega \\
\omega^{2} & \omega & 1
\end{pmatrix}, \label{Cjp1p} \\
C_{j}^{(2)}(x,t) = & \; f_{j,1}(x,t) \begin{pmatrix}
\omega^{2} & \omega & 1 \\
\omega^{2} & \omega & 1 \\
\omega^{2} & \omega & 1
\end{pmatrix} + f_{j,2}(x,t) \begin{pmatrix}
\omega^{2} & 1 & \omega \\
1 & \omega & \omega^{2} \\
\omega & \omega^{2} & 1
\end{pmatrix} \nonumber \\
& \; + f_{j,3}(x,t) \begin{pmatrix}
\omega^{2} & \omega^{2} & \omega^{2} \\
\omega & \omega & \omega \\
1 & 1 & 1
\end{pmatrix}, \label{Cjp2p}
\end{align}
where $\alpha_j(x,t)$, $\beta_j(x,t)$, $\gamma_j(x,t)$, $\delta_{j,l}(x,t)$, $\epsilon_{j,l}(x,t)$, $f_{j,l}(x,t)$, $j = 1,2,3$, $l=1,2,3$, are  real-valued functions of $x \geq 0$ and $0 \leq t \leq T$, and for $j=3$ these functions have rapid decay as $x \to +\infty$.
\end{itemize}
\end{proposition}

In \cite{CLgoodboussinesq}, it was enough to compute the coefficients $C_{j}^{(l)}$ for $l=-2,-1,0$. In this paper, we also need explicit expressions for $\smash{C_{j}^{(l)}}$ for $l=1,2$, see \eqref{Cjp1p}--\eqref{Cjp2p}. This is because the expressions for $M$ in Lemma \ref{M1XYlemma} below are more complicated (they involve more spectral functions, such as $S(k)$) than their analogs in \cite{CLgoodboussinesq}. 

\begin{proposition}[Asymptotics of $\mu_{j}^{A}$ as $k \to 0$]\label{XAYAat1prop}
Suppose $u_0, v_0 \in \mathcal{S}(\R_{+})$ and $\tilde{u}_{0}, \tilde{u}_{1}, \tilde{u}_{2}, \tilde{v}_0 \in C^{\infty}([0,T])$ and let $p \geq 0$ be an integer. 
Then there are $3 \times 3$-matrix valued functions $D_j^{(l)}(x,t)$, $j = 1,2,3$, $l = -2,-1, \dots, p$, with the following properties:
\begin{itemize}
\item For $x \geq 0$, $t\in [0,T]$ and $k$ within the domains of definition \eqref{lol4main adjoint} of $\mu_{j}^{A}$, the function $\mu_{j}^A$ satisfies
\begin{align}\label{XAYAat1a}
& \bigg|\frac{\partial^m}{\partial k^m}\big(\mu_{j}^A - I - \sum_{l=-2}^p D_1^{(l)}(x,t)k^l\big) \bigg| \leq
f_j(x,t)|k|^{p+1-m}, \qquad |k| \leq \frac{1}{2},
\end{align}
where $m \geq 0$ is any integer, $f_j(x,t)$ are smooth positive functions of $x \geq 0$ and $0 \leq t \leq T$, and $f_{3}(x,t)$ has rapid decay as $x \to +\infty$.

\item For each $l\geq -2$ and $j\in \{1,2,3\}$, $D_j^{(l)}(x,t)$ is a smooth function of $x \geq 0$ and $t \in [0,T]$, and $D_{3}^{(l)}(x,t)$ has rapid decay as $x \to +\infty$. 

\item The leading coefficients have the form
\begin{align}
D_{j}^{(-2)}(x,t)
= & \;
\tilde{\alpha}_{j}(x,t)\begin{pmatrix}
\omega & \omega & \omega \\
\omega^{2} & \omega^{2} & \omega^{2} \\
1 & 1 & 1 
\end{pmatrix}, \label{Djminus2} \\
D_{j}^{(-1)}(x,t) = & \;  \tilde{\beta}_{j}(x,t) \begin{pmatrix}
\omega^{2} & 1 & \omega \\
1 & \omega & \omega^{2} \\
\omega & \omega^{2} & 1 
\end{pmatrix} + \tilde{\gamma}_{j}(x,t) \begin{pmatrix}
\omega^{2} & \omega^{2} & \omega^{2} \\
\omega & \omega & \omega \\
1 & 1 & 1 
\end{pmatrix},  \label{Djminus1} \\
D_{j}^{(0)}(x,t) = & -I + \tilde{\delta}_{j,1}(x,t) \begin{pmatrix}
1 & \omega^{2} & \omega \\
\omega & 1 & \omega^{2} \\
\omega^{2} & \omega & 1
\end{pmatrix}  \nonumber \\
&  + \tilde{\delta}_{j,2}(x,t) \begin{pmatrix}
1 & \omega & \omega^{2} \\
\omega^{2} & 1 & \omega \\
\omega & \omega^{2} & 1
\end{pmatrix} + \tilde{\delta}_{j,3}(x,t)\begin{pmatrix}
1 & 1 & 1 \\
1 & 1 & 1 \\
1 & 1 & 1
\end{pmatrix}, \label{Djminus0} \\
D_{j}^{(1)}(x,t) = & \; \tilde{\epsilon}_{j,1}(x,t) \begin{pmatrix}
\omega & \omega & \omega \\
\omega^{2} & \omega^{2} & \omega^{2} \\
1 & 1 & 1
\end{pmatrix} + \tilde{\epsilon}_{j,2}(x,t) \begin{pmatrix}
\omega & 1 & \omega^{2} \\
1 & \omega^{2} & \omega \\
\omega^{2} & \omega & 1
\end{pmatrix} \nonumber \\
& + \tilde{\epsilon}_{j,3}(x,t) \begin{pmatrix}
\omega & \omega^{2} & 1 \\
\omega & \omega^{2} & 1 \\
\omega & \omega^{2} & 1
\end{pmatrix}, \label{Djp1p} \\
D_{j}^{(2)}(x,t) = & \; \tilde{f}_{j,1}(x,t) \begin{pmatrix}
\omega^{2} & 1 & \omega \\
1 & \omega & \omega^{2} \\
\omega & \omega^{2} & 1
\end{pmatrix} + \tilde{f}_{j,2}(x,t) \begin{pmatrix}
\omega^{2} & \omega^{2} & \omega^{2} \\
\omega & \omega & \omega \\
1 & 1 & 1
\end{pmatrix} \nonumber \\
& + \tilde{f}_{j,3}(x,t) \begin{pmatrix}
\omega^{2} & \omega & 1 \\
\omega^{2} & \omega & 1 \\
\omega^{2} & \omega & 1
\end{pmatrix}, \label{Djp2p}
\end{align}
where $\tilde{\alpha}_j(x,t)$, $\tilde{\beta}_j(x,t)$, $\tilde{\gamma}_j(x,t)$, $\tilde{\delta}_{j,l}(x,t)$, $\tilde{\epsilon}_{j,l}(x,t)$, $\tilde{f}_{j,l}(x,t)$, $j = 1,2,3$, $l=1,2,3$ are real-valued functions of $x \geq 0$ and $t\in [0,T]$, and for $j=3$ these functions have rapid decay as $x \to +\infty$. 
\end{itemize}
\end{proposition}

\subsection{The spectral functions $s(k)$ and $S(k)$}

Suppose temporarily that $u_0$ and $v_0$ are compactly supported. Then all columns of $\mu_{1},\mu_{2},\mu_{3}$ are analytic for $k \in \C\setminus\{0\}$, and since each of the three matrix-valued functions $\psi_{j}(x,t,k):=\mu_{j}(x,t,k)e^{x \mathcal{L}(k)+t \mathcal{Z}(k)}$, $j = 1,2,3$, satisfies the system
\begin{align*}
\begin{cases}
\psi_{x} = L \psi, \\
\psi_{t} = Z \psi,
\end{cases}
\end{align*}
they are related by some $(x,t)$-independent matrices, which we denote by $S(k)$ and $s(k)$, so that 
\begin{align}\label{mu3mu2mu1sS}
& \mu_{1}(x,t,k) = \mu_{2}(x,t,k)e^{x\hat{\mathcal{L}}(k)+t\hat{\mathcal{Z}}(k)}S(k), & & \mu_{3}(x,t,k) = \mu_{2}(x,t,k)e^{x\hat{\mathcal{L}}(k)+t\hat{\mathcal{Z}}(k)}s(k).
\end{align}
Evaluating both equations at $(x,t)=(0,0)$ and recalling that $\mu_{2}(0,0,k)=I$, we get
\begin{align}\label{def of S and s}
S(k) = \mu_1(0,0,k), \qquad  s(k) = \mu_3(0,0,k).
\end{align}

Suppose now that the initial values are not necessarily compactly supported. Then the second equation in \eqref{mu3mu2mu1sS} is not necessarily defined for all $k\in \C\setminus \{0\}$, but we nevertheless define $s$ as in \eqref{def of S and s}. 
Similarly, we define $S^A$ and $s^A$ by
\begin{align}\label{SAsAdef}
& S^{A}(k) = \mu_{1}^{A}(0,0,k), \qquad s^{A}(k) = \mu_{3}^{A}(0,0,k).
\end{align}
The next two propositions are a consequence of (\ref{def of S and s}), (\ref{SAsAdef}), and the properties of $\mu_1$ and $\mu_3$ established in Propositions \ref{XYprop}--\ref{XAYAat1prop}.

\begin{proposition}[Properties of $s(k)$ and $S(k)$]\label{sprop}
Suppose $u_0, v_0 \in \mathcal{S}(\R_{+})$ and $\tilde{u}_{0}, \tilde{u}_{1}, \tilde{u}_{2}, \tilde{v}_0 \in C^{\infty}([0,T])$. 
Then the spectral functions $s(k)$ and $S(k)$ have the following properties:
\begin{enumerate}[$(a)$]
\item The entries of $s(k)$ and $S(k)$ have the following domains of definition:
\begin{subequations}\label{sSboundednessregions}
\begin{align}
& s(k) \; \text{is well-defined and continuous for} \; k \in (\omega^{2} \overline{\mathcal{S}}, \omega \overline{\mathcal{S}}, \overline{\mathcal{S}})\setminus\{0\}, \label{region of boundedness of s}
	\\
& S(k) \; \text{is well-defined and continuous for} \; k \in \C \setminus\{0\}. \label{region of boundedness of S}
\end{align}
\end{subequations}
 
\item The entries of $s(k)$ and $S(k)$ are analytic in the interior of their domains of definition as given in \eqref{sSboundednessregions}. 
 
\item For $j = 1, 2, \dots$, the derivatives $\partial_k^js(k)$ and $\partial_k^j S(k)$ are well-defined and continuous for $k$ as in \eqref{sSboundednessregions}.

\item $s(k)$ and $S(k)$ obey the symmetries
\begin{align}\label{sym of s and S}
&  s(k) = \mathcal{A} s(\omega k)\mathcal{A}^{-1} = \mathcal{B} \overline{s(\overline{k})}\mathcal{B}, \qquad S(k) = \mathcal{A} S(\omega k)\mathcal{A}^{-1} = \mathcal{B} \overline{S(\overline{k})}\mathcal{B},
\end{align}

\item $s(k)$ and $S(k)$ approach the identity matrix as $k \to \infty$. More precisely, there are matrices $\{s_{l}\}_1^\infty$ and $\{S_{l}\}_1^\infty$ such that
\begin{align*}
& \Big|\partial_k^j \Big(s(k) - I - \sum_{{l}=1}^N \frac{s_{l}}{k^{l}}\Big)\Big| = O(k^{-N-1}), & & k \to \infty,
\ k \in (\omega^{2} \overline{\mathcal{S}}, \omega \overline{\mathcal{S}}, \overline{\mathcal{S}}), \\
& \Big|\partial_k^j \Big(S(k) - I - \sum_{{l}=1}^N \frac{S_{l}}{k^{l}}\Big)\Big| = O(k^{-N-1}), & & k \to \infty,
\ k \in (\omega^{2} \overline{\check{\mathcal{R}}}, \omega \overline{\check{\mathcal{R}}}, \overline{\check{\mathcal{R}}}),
\end{align*}
for $j = 0, 1, \dots, N$ and each integer $N \geq 1$.  

\item There exist matrices $\{s^{(p)}\}_{p=-2}^{+\infty}$ and $\{S^{(p)}\}_{p=-2}^{+\infty}$ such that, as $k \to 0$, 
\begin{align*}
& s(k) = \frac{s^{(-2)}}{k^{2}} + \frac{s^{(-1)}}{k} + s^{(0)} + s^{(1)}k + \ldots, & & k \in (\omega^{2} \overline{\mathcal{S}}, \omega \overline{\mathcal{S}}, \overline{\mathcal{S}}),  \\
& S(k) = \frac{S^{(-2)}}{k^{2}} + \frac{S^{(-1)}}{k} + S^{(0)} + S^{(1)}k + \ldots, & & k \in (\omega^{2} \overline{\check{\mathcal{R}}}, \omega \overline{\check{\mathcal{R}}}, \overline{\check{\mathcal{R}}}), 
\end{align*}
and the expansions can be differentiated termwise any number of times. The first matrices are given by
\begin{align}
& s^{(-2)} =  \mathfrak{s}^{(-2)} \begin{pmatrix}
\omega & \omega^{2} & 1 \\
\omega & \omega^{2} & 1 \\
\omega & \omega^{2} & 1 
\end{pmatrix}, \label{spm2p} \\
& s^{(-1)} =  \mathfrak{s}^{(-1)}_{1} \begin{pmatrix}
\omega^{2} & \omega & 1 \\
\omega^{2} & \omega & 1 \\
\omega^{2} & \omega & 1 
\end{pmatrix} + \mathfrak{s}^{(-1)}_{2} \begin{pmatrix}
0 & \omega^{2} & -\omega^{2} \\
-\omega & 0 & \omega \\
1 & -1 & 0 
\end{pmatrix}, \label{spm1p} \\
& s^{(0)} = I + \mathfrak{s}^{(0)}_{1} \begin{pmatrix}
1 & 1 & 1 \\
1 & 1 & 1 \\
1 & 1 & 1
\end{pmatrix} + \mathfrak{s}^{(0)}_{2} \begin{pmatrix}
0 & 1 & 1 \\
1 & 0 & 1 \\
1 & 1 & 0
\end{pmatrix} + \mathfrak{s}^{(0)}_{3} \begin{pmatrix}
0 & 1 & -1 \\
-1 & 0 & 1 \\
1 & -1 & 0
\end{pmatrix}, \label{sp0p} \\
& S^{(-2)} =  \mathscr{S}^{(-2)} \begin{pmatrix}
\omega & \omega^{2} & 1 \\
\omega & \omega^{2} & 1 \\
\omega & \omega^{2} & 1 
\end{pmatrix}, \label{Spm2p} \\
& S^{(-1)} =  \mathscr{S}^{(-1)}_{1} \begin{pmatrix}
\omega^{2} & \omega & 1 \\
\omega^{2} & \omega & 1 \\
\omega^{2} & \omega & 1 
\end{pmatrix} + \mathscr{S}^{(-1)}_{2} \begin{pmatrix}
0 & \omega^{2} & -\omega^{2} \\
-\omega & 0 & \omega \\
1 & -1 & 0 
\end{pmatrix}, \label{Spm1p} \\
& S^{(0)} = I + \mathscr{S}^{(0)}_{1} \begin{pmatrix}
1 & 1 & 1 \\
1 & 1 & 1 \\
1 & 1 & 1
\end{pmatrix} + \mathscr{S}^{(0)}_{2} \begin{pmatrix}
0 & 1 & 1 \\
1 & 0 & 1 \\
1 & 1 & 0
\end{pmatrix} + \mathscr{S}^{(0)}_{3} \begin{pmatrix}
0 & 1 & -1 \\
-1 & 0 & 1 \\
1 & -1 & 0
\end{pmatrix}, \label{Sp0p}
\end{align}
where
\begin{align*}
& \mathfrak{s}^{(-2)} := \int_{0}^{+\infty} \big( 2 u_{0} \gamma_{3} + (u_{0x}+v_{0})\delta_{3,3} \big)|_{t=0} dx \in \mathbb{R}, \\
& \mathfrak{s}^{(-1)}_{1} := \int_{0}^{+\infty} \big( 2 u_{0} \delta_{3,2} + (u_{0x}+v_{0})\epsilon_{3,3} \big)|_{t=0} dx \in \R, \\
& \mathfrak{s}^{(-1)}_{2} := -i \sqrt{3} \int_{0}^{+\infty} \big( 2 u_{0} \gamma_{3} + (u_{0x}+v_{0})\delta_{3,3} \big)|_{t=0} dx \in i\R, \\
& \mathfrak{s}^{(0)}_{1} \in \R, \qquad \mathfrak{s}^{(0)}_{2} \in \R, \qquad \mathfrak{s}^{(0)}_{3} \in i\R, \\
& \mathscr{S}^{(-2)} := \frac{1}{3} \int_{0}^{T} \big( 2 \tilde{u}_{0} \alpha_{1} + (\tilde{u}_{1}+3\tilde{v}_{0})\gamma_{1} + (3\tilde{v}_{1}-\tilde{u}_{2})\delta_{1,3} \big)|_{x=0} dt \in \mathbb{R}, \\
& \mathscr{S}^{(-1)}_{1} \in \mathbb{R}, \qquad \mathscr{S}^{(-1)}_{2} \in i\mathbb{R}, \\
& \mathscr{S}^{(0)}_{1} \in \mathbb{R}, \qquad \mathscr{S}^{(0)}_{2} \in \mathbb{R}, \qquad \mathscr{S}^{(0)}_{3} \in i\mathbb{R}.
\end{align*}

\item If $u_0, v_0$ have compact support, then $s(k)$ is defined and analytic for $k \in \C \setminus \{0\}$, $\det s(k) = 1$ for $k \in \C \setminus \{0\}$, and the second equation in \eqref{mu3mu2mu1sS} holds for $k \in \C \setminus \{0\}$.

\medskip \noindent 
Moreover, $S(k)$ is defined and analytic for $k \in \C \setminus \{0\}$, $\det S(k) = 1$ for $k \in \C \setminus \{0\}$, and the first equation in \eqref{mu3mu2mu1sS} holds for $k \in \C \setminus \{0\}$.
\end{enumerate}
\end{proposition}

By Proposition \ref{sprop}, $\det S(k) =1$ for all $k\in \C\setminus\{0\}$, and thus, by \eqref{mu3mu2mu1sS},
\begin{align}\label{global rel lol}
\mu_1(x,t,k) e^{x\hat{\mathcal{L}} + t \hat{\mathcal{Z}}}(S(k)^{-1}s(k)) = \mu_3(x,t,k), \qquad k \in (\omega^{2} \overline{\mathcal{S}}, \omega \overline{\mathcal{S}}, \overline{\mathcal{S}})\setminus\{0\}.
\end{align}
Evaluating \eqref{global rel lol} at $(0,T)$, using that $\mu_{1}(0,T,k) = I$, we find the global relation
\begin{align}\label{global rel}
S(k)^{-1}s(k) = e^{-T \hat{\mathcal{Z}}(k)}\mu_3(0,T,k), \qquad k \in (\omega^{2} \overline{\mathcal{S}}, \omega \overline{\mathcal{S}}, \overline{\mathcal{S}})\setminus\{0\}.
\end{align}
By taking limits as in \cite[Lemma 4.5]{CLgoodboussinesq}, we infer that the global relation \eqref{global rel} also holds for non-compactly supported initial data. This shows that the entries of $S^{-1}s$ have analytic continuations to $(\omega^{2} \mathcal{S}, \omega \mathcal{S}, \mathcal{S})$. 

\subsection{The spectral functions $s^A(k)$ and $S^A(k)$}
We define
\begin{align}\label{def of SA and sA}
& S^{A}(k) = I+\int_{0}^{T}e^{t\hat{\mathcal{Z}}(k)}(\mathsf{V}^{T}\mu_{1}^{A})(0,t,k)dt, & & s^{A}(k) = I+\int_{0}^{+\infty}e^{x\hat{\mathcal{L}}(k)}(\mathsf{U}^{T}\mu_{3}^{A})(x,0,k)dx.
\end{align}
and we have
\begin{align}\label{global rel adjoint}
S(k)^{T}s^{A}(k) = e^{T \hat{\mathcal{Z}}(k)}\mu_3^{A}(0,T,k), \qquad k \in (-\omega^{2} \overline{\mathcal{S}}, -\omega \overline{\mathcal{S}}, -\overline{\mathcal{S}})\setminus\{0\}.
\end{align}

\begin{proposition}[Properties of $s^A(k)$ and $S^A(k)$]\label{sAprop}
Suppose $u_0, v_0 \in \mathcal{S}(\R_{+})$ and $\tilde{u}_{0}, \tilde{u}_{1}, \tilde{u}_{2}, \tilde{v}_0 \in C^{\infty}([0,T])$. 
Then the spectral functions $s^A(k)$ and $S^A(k)$ have the following properties:
\begin{enumerate}[$(a)$]
\item The entries of $s(k)$ and $S(k)$ have the following domains of definition:
\begin{subequations}
\begin{align}
& s^{A}(k) \; \text{is well-defined and continuous for} \; k \in (-\omega^{2} \overline{\mathcal{S}}, -\omega \overline{\mathcal{S}}, -\overline{\mathcal{S}}), \label{region of boundedness of sA}
	\\
& S^{A}(k) \; \text{is well-defined and continuous for} \; k \in (\omega^{2}\overline{\check{\mathcal{T}}}, \omega \overline{\check{\mathcal{T}}},  \overline{\check{\mathcal{T}}}). \label{region of boundedness of SA}
\end{align}
\end{subequations}

\item The entries of $s^A(k)$ and $S^A(k)$ are analytic in the interior of their domains of definition as given in \eqref{region of boundedness of sA}--\eqref{region of boundedness of SA}. 
 
\item For $j = 1, 2, \dots$, the derivatives $\partial_k^js^A(k)$ and $\partial_k^j S^A(k)$ are well-defined and continuous for $k$ in \eqref{region of boundedness of sA}--\eqref{region of boundedness of SA}.

\item $s^A(k)$ and $S^A(k)$ obey the symmetries
\begin{align}\label{sym of sA and SA}
&  s^A(k) = \mathcal{A} s^A(\omega k)\mathcal{A}^{-1} = \mathcal{B} \overline{s^A(\overline{k})}\mathcal{B}, \qquad S^A(k) = \mathcal{A} S^A(\omega k)\mathcal{A}^{-1} = \mathcal{B} \overline{S^A(\overline{k})}\mathcal{B}.
\end{align}

\item $s^{A}(k)$ and $S^{A}(k)$ approach the identity matrix as $k \to \infty$. More precisely, there are matrices $\{s^{A}_j\}_1^\infty$ and $\{S^{A}_j\}_1^\infty$ such that
\begin{align*}
& \Big|\partial_k^j \Big(s^{A}(k) - I - \sum_{j=1}^N \frac{s^{A}_j}{k^j}\Big)\Big| = O(k^{-N-1}), & & k \to \infty,
\ k \in (-\omega^{2} \overline{\mathcal{S}}, -\omega \overline{\mathcal{S}}, -\overline{\mathcal{S}}), \\
& \Big|\partial_k^j \Big(S^{A}(k) - I - \sum_{j=1}^N \frac{S^{A}_j}{k^j}\Big)\Big| = O(k^{-N-1}), & & k \to \infty,
\ k \in (\omega^{2} \overline{\check{\mathcal{T}}}, \omega \overline{\check{\mathcal{T}}}, \overline{\check{\mathcal{T}}}),
\end{align*}
for $j = 0, 1, \dots, N$ and each integer $N \geq 1$. 

\item There exist matrices $\{s^{A(p)}\}_{p=-2}^{+\infty}$ and $\{S^{A(p)}\}_{p=-2}^{+\infty}$ such that, as $k \to 0$,
\begin{align}
& s^{A}(k) = \frac{s^{A(-2)}}{k^{2}} + \frac{s^{A(-1)}}{k} + s^{A(0)} + s^{A(1)}k + \ldots, & &  k \in (-\omega^{2} \overline{\mathcal{S}}, -\omega \overline{\mathcal{S}}, -\overline{\mathcal{S}}), \label{sA at 0} \\
& S^{A}(k) = \frac{S^{A(-2)}}{k^{2}} + \frac{S^{A(-1)}}{k} + S^{A(0)} + S^{A(1)}k + \ldots, & &  k \in (\omega^{2} \overline{\check{\mathcal{T}}}, \omega \overline{\check{\mathcal{T}}}, \overline{\check{\mathcal{T}}}), \label{SA at 0}
\end{align}
and the expansions can be differentiated termwise any number of times. The first matrices are given by
\begin{align}
& s^{A(-2)} =  \mathfrak{s}^{A(-2)} \begin{pmatrix}
\omega & \omega & \omega \\
\omega^{2} & \omega^{2} & \omega^{2} \\
1 & 1 & 1 
\end{pmatrix}, \label{sApm2p} \\
& s^{A(-1)} =  \mathfrak{s}^{A(-1)}_{1} \begin{pmatrix}
\omega^{2} & 1 & \omega \\
1 & \omega & \omega^{2} \\
\omega & \omega^{2} & 1 
\end{pmatrix} + \mathfrak{s}^{A(-1)}_{2} \begin{pmatrix}
0 & \omega & -1 \\
-\omega^{2} & 0 & 1 \\
\omega^{2} & -\omega & 0 
\end{pmatrix}, \label{sApm1p} \\
& s^{A(0)} = I +  \mathfrak{s}^{A(0)}_{1} \begin{pmatrix}
1 & \omega^{2} & \omega \\
\omega & 1 & \omega^{2} \\
\omega^{2} & \omega & 1 
\end{pmatrix} + \mathfrak{s}^{A(0)}_{2} \begin{pmatrix}
0 & 1 & 1 \\
1 & 0 & 1 \\
1 & 1 & 0 
\end{pmatrix} \nonumber \\
& \hspace{1.3cm}  + \mathfrak{s}^{A(0)}_{3} \begin{pmatrix}
0 & 1 & -1 \\
-1 & 0 & 1 \\
1 & -1 & 0 
\end{pmatrix}, \label{sAp0p} \\
& S^{A(-2)} =  \mathscr{S}^{A(-2)} \begin{pmatrix}
\omega & \omega & \omega \\
\omega^{2} & \omega^{2} & \omega^{2} \\
1 & 1 & 1 
\end{pmatrix}, \label{SApm2p} \\
& S^{A(-1)} =  \mathscr{S}^{A(-1)}_{1} \begin{pmatrix}
\omega^{2} & 1 & \omega \\
1 & \omega & \omega^{2} \\
\omega & \omega^{2} & 1 
\end{pmatrix} + \mathscr{S}^{A(-1)}_{2} \begin{pmatrix}
0 & \omega & -1 \\
-\omega^{2} & 0 & 1 \\
\omega^{2} & -\omega & 0 
\end{pmatrix}, \label{SApm1p} \\
& S^{A(0)} = I +  \mathscr{S}^{A(0)}_{1} \begin{pmatrix}
1 & \omega^{2} & \omega \\
\omega & 1 & \omega^{2} \\
\omega^{2} & \omega & 1 
\end{pmatrix} + \mathscr{S}^{A(0)}_{2} \begin{pmatrix}
0 & 1 & 1 \\
1 & 0 & 1 \\
1 & 1 & 0 
\end{pmatrix} \nonumber \\
& \hspace{1.3cm} + \mathscr{S}^{A(0)}_{3} \begin{pmatrix}
0 & 1 & -1 \\
-1 & 0 & 1 \\
1 & -1 & 0 
\end{pmatrix}, \label{SAp0p}
\end{align}
where
\begin{align*}
& \mathfrak{s}^{A(-2)} := -\int_{0}^{+\infty}  \big((u_{0x}+v_{0})\tilde{\delta}_{3,3}\big)|_{t=0} dx \in \mathbb{R}, \\
& \mathfrak{s}^{A(-1)}_{1} := -\int_{0}^{+\infty} \big( 2 u_{0} \tilde{\delta}_{3,3} + (u_{0x}+v_{0})\tilde{\epsilon}_{3,3} \big)|_{t=0} dx \in \R, \\
& \mathfrak{s}^{A(-1)}_{2} := -i \sqrt{3} \int_{0}^{+\infty} \big( (2 u_{0} + (u_{0x}+v_{0})x)\tilde{\delta}_{3,3} \big)|_{t=0} dx \in i\R, \\
& \mathfrak{s}^{A(0)}_{1} \in \R, \qquad \mathfrak{s}^{A(0)}_{2} \in \R, \qquad \mathfrak{s}^{A(0)}_{3} \in i\R, \\
& \mathscr{S}^{A(-2)} := \frac{1}{3} \int_{0}^{T} \big( 4\tilde{u}_{0}\tilde{\alpha}_{1} + (\tilde{u}_{1} - 3\tilde{v}_{0})\tilde{\gamma}_{3} + (\tilde{u}_{2} - 3 \tilde{u}_{0}')\tilde{\delta}_{1,3} \big)|_{x=0} dt \in \mathbb{R}, \\
& \mathscr{S}^{A(-1)}_{1} \in \mathbb{R}, \qquad \mathscr{S}^{A(-1)}_{2} \in i\mathbb{R}, \\
& \mathscr{S}^{A(0)}_{1} \in \mathbb{R}, \qquad \mathscr{S}^{A(0)}_{2} \in \mathbb{R}, \qquad \mathscr{S}^{A(0)}_{3} \in i\mathbb{R}.
\end{align*}
\item If $u_0, v_0$ have compact support, then $s^A(k)$ is defined and equals the cofactor matrix of $s(k)$ for all $k \in \C \setminus \{0\}$.

\medskip \noindent If $\tilde{u}_{0}, \tilde{u}_{1}, \tilde{u}_{2}, \tilde{v}_0$ have compact support, then $S^{A}(k)$ is defined and equals the cofactor matrix of $S(k)$ for all $k \in \C \setminus \{0\}$. 
\end{enumerate}

\end{proposition}

\subsection{Proof of Theorem \ref{r1r2th}}\label{r1r2subsec}
The theorem follows from Propositions \ref{sprop} and \ref{sAprop}. Indeed, recall from (\ref{r1r2def}) that
\begin{align*}
& r_1(k) = \frac{(S^{-1}s)_{12}}{(S^{-1}s)_{11}}, & & r_2(k) = \frac{s^A_{12}S^{A}_{33}-s^{A}_{32}S^{A}_{13}}{s^A_{11}S^{A}_{33}-s^{A}_{31}S^{A}_{13}}, \\
& r_{3}(k) = \frac{S^{A}_{31}}{s^{A}_{33}} \frac{1}{(S^{-1}s)_{11}}, & & r_{4}(k) = \frac{S^{A}_{31}}{s^{A}_{33}} \frac{s_{21}}{s^{A}_{33}S^{A}_{11}-s^{A}_{13}S^{A}_{31}},
\end{align*}
where $S^{-1}$ is defined as the transpose of $S^{A}$. Statements $(a)$ and $(c)$ of Propositions \ref{sprop} and \ref{sAprop} imply that 
$r_{j}(k)$, $j=1,2,3,4$, are smooth on their domain of definitions, except at possible values of $k$ where a denominator vanishes. By Assumption \ref{solitonlessassumption}, the denominators of $r_{1}, r_{3}$, and $r_{4}$ do not vanish. Also, the symmetries \eqref{sym of sA and SA} imply
\begin{align*}
(s^{A}_{33}S^{A}_{11}-s^{A}_{13}S^{A}_{31})(k) = \overline{(s^A_{11}S^{A}_{33}-s^{A}_{31}S^{A}_{13})(\overline{\omega k})},
\end{align*}
and thus the denominator of $r_{2}$ also does not vanish. Thus the functions $r_{j}$ are smooth on their domain of definition, which proves $(i)$.
Moreover, statement $(e)$ of the same propositions imply that $r_1(k)$ and $r_2(k)$ satisfy (\ref{r1r2r3r4 slow decay}) with
\begin{align*}
r_{1}^{(1)} = s_{1,12}+S^{A}_{1,21}, \qquad r_{2}^{(1)} = s^{A}_{1,12}, \qquad r_{3}^{(1)} = S^{A}_{1,31}, \qquad r_{4}^{(2)} = s_{1,21}S^{A}_{1,31},
\end{align*}
and combining the above with the symmetries \eqref{sym of s and S} and \eqref{sym of sA and SA}, we find \eqref{lol3}. This proves $(iv)$. 
 
Assumption \ref{originassumption} implies that the coefficients $\mathfrak{s}^{(-2)}, \mathscr{S}^{(-2)}, \mathfrak{s}^{A(-2)}, \mathscr{S}^{A(-2)}, \mathscr{S}^{A(-1)}_{1}$ and $\mathscr{S}^{A(-1)}_{2}$ in (\ref{spm2p}), \eqref{Spm2p}, (\ref{sApm2p}), \eqref{SApm2p} and \eqref{SApm1p} are all nonzero. 
Hence properties $(ii)$ and $(iii)$ of Theorem \ref{r1r2th} related to the behavior of $r_j(k)$, $j=1,2,3,4$ as $k \to 0$ follow from a long but direct computation using statement $(f)$ of Propositions \ref{sprop} and \ref{sAprop}.

It remains to prove that $|r_{1}(k)|<1$ for $k > 0$ and that $|r_{2}(k)| <1$ for $k < 0$. By \eqref{global rel}, for $k>0$, all four entries of $\{(S^{-1}s)_{ij}\}_{i,j=1}^{2}$ are well-defined, and hence $(S^{-1}s)_{11}(k)(S^{-1}s)_{22}(k)-(S^{-1}s)_{12}(k)(S^{-1}s)_{21}(k) = (S^{T}s^{A})_{33}(k)$. Using also the symmetries \eqref{sym of s and S} and \eqref{sym of sA and SA}, we get
\begin{align}
1-|r_{1}(k)|^{2} & = 1 - \bigg| \frac{(S^{-1}s)_{12}(k)}{(S^{-1}s)_{11}(k)} \bigg|^{2} = \frac{(S^{-1}s)_{11}(k)(S^{-1}s)_{22}(k)-(S^{-1}s)_{12}(k)(S^{-1}s)_{21}(k)}{|(S^{-1}s)_{11}(k)|^{2}} \nonumber \\
& = \frac{(S^{T}s^{A})_{33}(k)}{|(S^{-1}s)_{11}(k)|^{2}}. \label{lol2 bis}
\end{align}
Since the left-hand side of \eqref{lol2 bis} is real and tends to $1$ as $k \to +\infty$, and since the right-hand side is non-zero for all $k > 0$ by assumption, we deduce that $1-|r_{1}(k)|^{2}>0$ for all $k > 0$.

By \eqref{region of boundedness of s} and \eqref{region of boundedness of s}, the first two columns of $s^{A}(k)$ and the third column of $s(k)$ are well-defined for $k<0$. Thus, for $k<0$,
\begin{align*}
& s^{A}_{21}s^{A}_{32}-s^{A}_{22}s^{A}_{31} = s_{13}, \\
& s^{A}_{12}s^{A}_{31}-s^{A}_{11}s^{A}_{32} = s_{23}, \\
& s^{A}_{11}s^{A}_{22}-s^{A}_{21}s^{A}_{32} = s_{33},
\end{align*}
and by combining these relations with \eqref{sym of s and S} and \eqref{sym of sA and SA}, we obtain
\begin{align}
1-|r_{2}(k)|^{2} & = 1 - \frac{(s^A_{12}S^{A}_{33}-s^{A}_{32}S^{A}_{13})(s^A_{21}S^{A}_{33}-s^{A}_{31}S^{A}_{23})}{(s^A_{11}S^{A}_{33}-s^{A}_{31}S^{A}_{13})(s^A_{22}S^{A}_{33}-s^{A}_{32}S^{A}_{23})}(k) \nonumber \\
& = \frac{S^{A}_{33}(s_{13}S^{A}_{13}+s_{23}S^{A}_{23}+s_{33}S^{A}_{33})}{|s^A_{11}S^{A}_{33}-s^{A}_{31}S^{A}_{13}|^{2}}(k) = \frac{S^{A}_{33}(S^{-1}s)_{33}}{|s^A_{11}S^{A}_{33}-s^{A}_{31}S^{A}_{13}|^{2}}(k) \nonumber \\
& = \frac{S^{A}_{33}(k)(S^{-1}s)_{11}(\omega^{2}k)}{|s^A_{11}S^{A}_{33}-s^{A}_{31}S^{A}_{13}|^{2}(k)} \label{lol2}
\end{align}
for $k < 0$. Since the left-hand side of \eqref{lol2} is real and tends to $1$ as $k \to -\infty$, and since the right-hand side is non-zero for all $k < 0$ by Assumption \ref{solitonlessassumption} and the assumption made in Theorem \ref{r1r2th} $(v)$, we deduce that $1-|r_{2}(k)|^{2}>0$ for all $k < 0$.

\section{The function $M$}
For each $n = 1, \dots, 12$, we define a solution $M_n(x,t,k)$ of (\ref{Xlax}) by the following system of Fredholm integral equations: 
\begin{equation}\label{Mndef}
(M_n)_{ij}(x,t,k) = \delta_{ij} + \int_{\gamma_{ij}^n} \left(e^{\hat{\mathcal{L}}x + \hat{\mathcal{Z}}t} W(x',t',k)\right)_{ij}, \qquad k \in D_n, \quad i,j = 1, 2,3,
\end{equation}
where the contours $\gamma^n_{ij}$, $n = 1, \dots, 12$, $i,j = 1, 2,3$, are defined by
 \begin{align} \label{gammaijnudef}
 \gamma_{ij}^n =  \begin{cases}
 \gamma_1,  \qquad \text{Re}\, l_i(k) < \text{Re}\, l_j(k), \quad \text{Re}\, z_i(k) \geq \text{Re}\, z_j(k),
	\\
\gamma_2,  \qquad \text{Re}\, l_i(k) < \text{Re}\, l_j(k),\quad \text{Re}\, z_i(k) < \text{Re}\, z_j(k),
	\\
\gamma_3,  \qquad \text{Re}\, l_i(k) \geq \text{Re}\, l_j(k),
	\\
\end{cases} \quad \text{for} \quad k \in D_n,
\end{align}
and $W$ is given by (\ref{Wdef}) with $\mu$ replaced by $M_n$. 

As the next proposition shows, this makes all entries of $M_n$ well-defined for $k \in \bar{D}_n\setminus \mathcal{Q}$, where 
\begin{align}\label{calQdef}
  \mathcal{Q} = \{0\} \cup \mathsf{Z}
\end{align}
and $\mathsf{Z}$ denotes the set of zeros of the Fredholm determinants associated with (\ref{Mndef}). 

\begin{proposition}[Basic properties of $M_n$]\label{Mnprop}
Suppose $u_0, v_0 \in \mathcal{S}(\R_{+})$ and $\tilde{u}_{0}, \tilde{u}_{1}, \tilde{u}_{2}, \tilde{v}_0 \in C^{\infty}([0,T])$. 
Then the equations (\ref{Mndef}) uniquely define six $3 \times 3$-matrix valued solutions $\{M_n\}_1^6$ of (\ref{Xlax}) with the following properties:
\begin{enumerate}[$(a)$]
\item The function $M_n(x,t, k)$ is defined for $x \geq 0$, $t\in [0,T]$ and $k \in \bar{D}_n \setminus \mathcal{Q}$. For each $k \in \bar{D}_n  \setminus \mathcal{Q}$ and $t\in [0,T]$, $M_n(\cdot,t, k)$ is smooth and satisfies the $x$-part of (\ref{Xlax}). For each $k \in \bar{D}_n  \setminus \mathcal{Q}$ and $x \geq 0$, $M_n(x,\cdot, k)$ is smooth and satisfies the $t$-part of (\ref{Xlax}). 

\item For each $x \geq 0$ and $t\in [0,T]$, the function $M_n(x,t,\cdot)$ is continuous for $k \in \bar{D}_n \setminus \mathcal{Q}$ and analytic for $k \in D_n\setminus \mathcal{Q}$.

\item For each $\epsilon > 0$, there exists a $C = C(\epsilon)$ such that
\begin{align}\label{Mnbounded}
|M_n(x,t,k)| \leq C, \qquad x \geq 0, \; t\in[0,T], \ k \in \bar{D}_n, \ \mbox{dist}(k, \mathcal{Q}) \geq \epsilon.
\end{align}

\item For each $x \geq 0$, $t\in [0,T]$ and each $j = 1, 2, \dots$, the partial derivative $\frac{\partial^j M_n}{\partial k^j}(x,t, \cdot)$ has a continuous extension to $\bar{D}_n \setminus \mathcal{Q}$.

\item $\det M_n(x,t,k) = 1$ for $x \geq 0$, $t\in [0,T]$ and $k \in \bar{D}_n \setminus \mathcal{Q}$.

\item For each $x \geq 0$ and $t\in [0,T]$, the sectionally analytic function $M(x,t,k)$ defined by $M(x,t,k) = M_n(x,t,k)$ for $k \in D_n$ satisfies the symmetries
\begin{align}\label{Msymm}
 M(x,t, k) = \mathcal{A} M(x,t,\omega k)\mathcal{A}^{-1} = \mathcal{B} \overline{M(x,t,\overline{k})}\mathcal{B}, \qquad k \in \C \setminus \mathcal{Q}.
\end{align}

\end{enumerate}
\end{proposition}

\begin{lemma}[Asymptotics of $M$ as $k \to \infty$]\label{Matinftylemma}
Suppose $u_0, v_0 \in \mathcal{S}(\R_{+})$ and $\tilde{u}_{0}, \tilde{u}_{1}, \tilde{u}_{2}, \tilde{v}_0 \in C^{\infty}([0,T])$. For any integer $p \geq 1$, there exist $R > 0$ and $C>0$ such that
\begin{align*}
& \bigg|M(x,t,k) - \bigg(I + \frac{\mu_{3}^{(1)}(x,t)}{k} + \cdots + \frac{\mu_{3}^{(p)}(x,t)}{k^{p}}\bigg) \bigg| \leq
\frac{C}{|k|^{p+1}}, \quad x \geq 0, \, t\in [0,T], \   |k| \geq R,
\end{align*}
with $k \in \C \setminus \Gamma$. The coefficients $\mu_{3}^{(1)}$ and $\mu_{3}^{(2)}$ were computed explicitly in Proposition \ref{prop:first two coeff at infty}.
\end{lemma}

\begin{lemma}[Relation between $M_n$ and $\mu_{j}$]\label{Snexplicitlemma}
Suppose the functions $u_0, v_0 \in \mathcal{S}(\R_{+})$ and $\tilde{u}_{0}, \tilde{u}_{1}, \tilde{u}_{2}, \tilde{v}_0 \in C^{\infty}([0,T])$ have compact support. Then
\begin{align*}
M_n(x,t,k) & = \mu_{1}(x,t, k) e^{x\hat{\mathcal{L}}(k)+t\hat{\mathcal{Z}}(k)} R_n(k) \\
& = \mu_{2}(x,t, k) e^{x\hat{\mathcal{L}}(k)+t\hat{\mathcal{Z}}(k)} S_n(k) \\
& = \mu_{3}(x,t, k) e^{x\hat{\mathcal{L}}(k)+t\hat{\mathcal{Z}}(k)} T_n(k), \quad x \geq 0, \; t\in [0,T], \ k \in \bar{D}_n\setminus \mathcal{Q}, \ n = 1, \dots, 12,
\end{align*}
where $S_n(k)=s(k)T_n(k)$, $R_n(k)=S(k)^{-1}s(k)T_{n}(k)=S(k)^{-1}S_{n}(k)$, and the $T_{n}(k)$ for $n\in \{1,2,6,7,12\}$ are given by
\begin{align}
&  T_1(k) = \begin{pmatrix}
  1 & -\frac{(S^{-1}s)_{12}}{(S^{-1}s)_{11}} & \frac{s^{A}_{31}}{s^{A}_{33}} \\[0.2cm]
 0 & 1 & \frac{s^{A}_{32}}{s^{A}_{33}} \\
 0 & 0 & 1
   \end{pmatrix},
\qquad
  T_2(k) =  \begin{pmatrix}
  1 & -\frac{(S^{-1}s)_{12}}{(S^{-1}s)_{11}} & \frac{s^{A}_{31}S^{A}_{11}-s^{A}_{11}S^{A}_{31}}{s^{A}_{33}S^{A}_{11}-s^{A}_{13}S^{A}_{31}} \\[0.2cm]
 0 & 1 & \frac{s^{A}_{32}S^{A}_{11}-s^{A}_{12}S^{A}_{31}}{s^{A}_{33}S^{A}_{11}-s^{A}_{13}S^{A}_{31}} \\
 0 & 0 & 1
   \end{pmatrix}, \nonumber \\
& T_{12}(k) = \begin{pmatrix}
1 & 0 & \frac{s^{A}_{31}}{s^{A}_{33}}, \\[0.2cm]
- \frac{(S^{-1}s)_{21}}{(S^{-1}s)_{22}} & 1 & \frac{s^{A}_{32}}{s^{A}_{33}} \\[0.2cm]
0 & 0 & 1
\end{pmatrix}, \quad T_{6}(k) = \begin{pmatrix}
1 & \frac{s^{A}_{21}S^{A}_{33}-s^{A}_{31}S^{A}_{23}}{s^{A}_{22}S^{A}_{33}-s^{A}_{32}S^{A}_{23}} & 0 \\[0.2cm] 
0 & 1 & 0 \\
- \frac{(S^{-1}s)_{31}}{(S^{-1}s)_{33}} & \frac{s^{A}_{23}S^{A}_{33}-s^{A}_{33}S^{A}_{23}}{s^{A}_{22}S^{A}_{33}-s^{A}_{32}S^{A}_{23}} & 1
\end{pmatrix}, \nonumber \\
& T_{7}(k) = \begin{pmatrix}
1 & 0 & 0 \\
\frac{s^{A}_{12}S^{A}_{33}-s^{A}_{32}S^{A}_{13}}{s^{A}_{11}S^{A}_{33}-s^{A}_{31}S^{A}_{13}} & 1 & 0 \\[0.2cm]
\frac{s^{A}_{13}S^{A}_{33}-s^{A}_{33}S^{A}_{13}}{s^{A}_{11}S^{A}_{33}-s^{A}_{31}S^{A}_{13}} & -\frac{(S^{-1}s)_{32}}{(S^{-1}s)_{33}} & 1
\end{pmatrix}. \label{SnTnexplicit}
\end{align}
For the other values of $n$, $T_{n}$ can be obtained from \eqref{SnTnexplicit} using the $\mathcal{A}$- and $\mathcal{B}$-symmetries.
\end{lemma}
\proofbegin
Choose $K > 0$ such that $u_0, v_0$ have support in $[0,K]$. 
Define $R_{n}(k)$, $S_n(k)$ and $T_n(k)$, $n = 1, \dots, 12$, by
\begin{align}\label{SnTndef}
\begin{cases}
R_n(k) = \lim_{t\to T}e^{-t\hat{\mathcal{Z}}(k)}M_n(0,t,k), \\
S_n(k) = M_n(0,0,k), \\
T_n(k) =  \displaystyle{\lim_{x \to \infty}} e^{-x\hat{\mathcal{L}}(k)}M_n(x,0,k), 
\end{cases}
\quad k \in \bar{D}_n\setminus \mathcal{Q}, 
\end{align}
where the limits exist because $\mathsf{U}(x,0,k) = 0$ for $|x| > K$, which implies by \eqref{Mndef} that $e^{-x\hat{\mathcal{L}}}M_n(x,0,k)$ is independent of $x$ for $|x| > K$.
Recall that $\mu_{j}(x,t,k)$, and $s(k)$ are defined for all $k \in \C \setminus \{0\}$ for compactly supported data. We find 
\begin{align}
M_n(x,t,k) = \mu_{1}(x,t,k) e^{x\hat{\mathcal{L}}+t\hat{\mathcal{Z}}} R_n(k) = \mu_{2}(x,t,k) e^{x\hat{\mathcal{L}}+t\hat{\mathcal{Z}}} S_n(k) = \mu_{3}(x,t,k) e^{x\hat{\mathcal{L}}+t\hat{\mathcal{Z}}} T_n(k),
\end{align}   
and hence, comparing with (\ref{mu3mu2mu1sS}), 
\begin{equation}\label{sSSnrelations}  
s(k) = S_n(k)T_n(k)^{-1}, \quad S(k)=S_{n}(k)R_{n}(k)^{-1}, \qquad k \in \bar{D}_n\setminus \mathcal{Q}.
\end{equation}
Given $s(k)$, equation (\ref{sSSnrelations}) constitutes a matrix factorization problem which can be uniquely solved for $S_n(k)$ and $T_n(k)$.
In fact, the integral equations (\ref{Mndef}) imply that
\begin{align} 
& \left(R_n(k)\right)_{ij} = 0 \quad \text{if} \quad \gamma_{ij}^n = \gamma_{1}, \nonumber \\
& \left(S_n(k)\right)_{ij} = 0 \quad \text{if} \quad \gamma_{ij}^n = \gamma_{2}, \nonumber \\
& \left(T_n(k)\right)_{ij} = \delta_{ij} \quad \text{if} \quad \gamma_{ij}^n = \gamma_{3}, \nonumber
\end{align}
so the relation (\ref{sSSnrelations}) yields $18$ scalar equations for $18$ unknowns. 
The explicit solution of this algebraic system gives (\ref{SnTnexplicit}).
\proofend

\begin{lemma}[Jump condition for $M$]\label{Mjumplemma}
Let $u_0, v_0 \in \mathcal{S}(\R_{+})$ and $\tilde{u}_{0}, \tilde{u}_{1}, \tilde{u}_{2}, \tilde{v}_0 \in C^{\infty}([0,T])$. For each $x \geq 0$, $t\in [0,T]$, $M(x,t,k)$ satisfies the jump condition
\begin{align*}
  M_+(x,t,k) = M_-(x,t, k) v(x, t, k), \qquad k \in \Gamma \setminus \mathcal{Q},
\end{align*}
where $v$ is the jump matrix defined in (\ref{vdef}) and $\mathcal{Q}$ is the set defined in (\ref{calQdef}).
\end{lemma}
\begin{proof}
We will show that 
\begin{align}\label{M1M6v1}
M_1 = M_{12} v_1, \qquad k \in \Gamma_{1} \setminus \mathcal{Q};
\end{align}
the proof that $M_{n} = M_{n-1} v_{n}$ for $k \in \Gamma_{n}\setminus \mathcal{Q}$, $n=2,\ldots,12$ is similar. 

Suppose first that $u_0, v_0$ have support in some compact subset $[0,K]$ for some $K > 0$. 
For each $k$, $M_n(x,t,k)$ is a smooth function of $x \geq 0$ and $t\in [0,T]$ which satisfies (\ref{Xlax}). 
Hence there exists a matrix $J_1(k)$ independent of $x$ and $t$ such that
\begin{align}\label{M1M6J1}
M_1(x,t,k) = M_{12}(x,t,k) e^{x\widehat{\mathcal{L}(k)}+t\widehat{\mathcal{Z}(k)}}J_1(k), \qquad k \in \Gamma_{1}\setminus \mathcal{Q}.
\end{align}
For $x>K$ and $t=0$ we have $M_n(x,0,k) = e^{x\widehat{\mathcal{L}(k)}}T_n(k)$, where $T_n(k)$ is the matrix defined in (\ref{SnTndef}). Hence, evaluation of (\ref{M1M6J1}) at $x > K$ and $t=0$ gives
$$J_1(k) = T_{12}(k)^{-1}T_1(k).$$
Hence, by (\ref{vdef}) and Lemma \ref{Snexplicitlemma},
$$e^{x\hat{\mathcal{L}}+t\hat{\mathcal{Z}}}[T_{12}(k)^{-1}T_1(k)] = e^{x\hat{\mathcal{L}}+t\hat{\mathcal{Z}}}  \begin{pmatrix}  
 1 & - r_1(k) & 0 \\
  r_1^*(k) & 1 - |r_1(k)|^2 & 0 \\
  0 & 0 & 1
  \end{pmatrix} = v_1(x,t,k).$$
This completes the proof of (\ref{M1M6v1}) for compactly supported $u_0,v_0$. We can then extend (\ref{M1M6v1}) to non-compactly supported $u_0,v_0$ as in \cite[Proposition 2.6]{L3x3} or \cite[Lemma 4.5]{CLgoodboussinesq}. 
\end{proof}

\begin{lemma}\label{M1XYlemma}
Let $u_0, v_0 \in \mathcal{S}(\R_{+})$ and $\tilde{u}_{0}, \tilde{u}_{1}, \tilde{u}_{2}, \tilde{v}_0 \in C^{\infty}([0,T])$.
The functions $M_1$ and $M_{2}$ can be expressed in terms of the entries of $\mu_{1},\mu_{2},\mu_{3},s,S,\mu_{1}^{A},\mu_{2}^{A},\mu_{3}^{A},s^{A},S^{A}$ as follows:
\begin{align}
& M_1 = \begin{pmatrix} 
(\mu_{3})_{11} & (\mu_{2})_{11}\frac{s^{A}_{23}S^{A}_{31}-s^{A}_{33}S^{A}_{21}}{(S^{-1}s)_{11}}e^{-\theta_{21}}+(\mu_{2})_{12}\frac{s^{A}_{33}S^{A}_{11}-s^{A}_{13}S^{A}_{31}}{(S^{-1}s)_{11}}+(\mu_{2})_{13}\frac{s^{A}_{13}S^{A}_{21}-s^{A}_{23}S^{A}_{11}}{(S^{-1}s)_{11}}e^{\theta_{32}} & \frac{(\mu_{2})_{13}}{s_{33}^A} \\
(\mu_{3})_{21} & (\mu_{2})_{21}\frac{s^{A}_{23}S^{A}_{31}-s^{A}_{33}S^{A}_{21}}{(S^{-1}s)_{11}}e^{-\theta_{21}}+(\mu_{2})_{22}\frac{s^{A}_{33}S^{A}_{11}-s^{A}_{13}S^{A}_{31}}{(S^{-1}s)_{11}}+(\mu_{2})_{23}\frac{s^{A}_{13}S^{A}_{21}-s^{A}_{23}S^{A}_{11}}{(S^{-1}s)_{11}}e^{\theta_{32}} & \frac{(\mu_{2})_{23}}{s_{33}^A} \\
(\mu_{3})_{31} & (\mu_{2})_{31}\frac{s^{A}_{23}S^{A}_{31}-s^{A}_{33}S^{A}_{21}}{(S^{-1}s)_{11}}e^{-\theta_{21}}+(\mu_{2})_{32}\frac{s^{A}_{33}S^{A}_{11}-s^{A}_{13}S^{A}_{31}}{(S^{-1}s)_{11}}+(\mu_{2})_{33}\frac{s^{A}_{13}S^{A}_{21}-s^{A}_{23}S^{A}_{11}}{(S^{-1}s)_{11}}e^{\theta_{32}} & \frac{(\mu_{2})_{33}}{s_{33}^A} 
\end{pmatrix}, \label{M1 in terms of spectral functions}
\end{align}
for all $x \geq 0$, $t\in [0,T]$, and $k \in \bar{D}_1 \setminus \mathcal{Q}$, and
\begin{align*}
& M_2 = \left(\begin{array}{l l l}
(\mu_{3})_{11} & (\mu_{2})_{11}\frac{s^{A}_{23}S^{A}_{31}-s^{A}_{33}S^{A}_{21}}{(S^{-1}s)_{11}}e^{-\theta_{21}}+(\mu_{2})_{12}\frac{s^{A}_{33}S^{A}_{11}-s^{A}_{13}S^{A}_{31}}{(S^{-1}s)_{11}}+(\mu_{2})_{13}\frac{s^{A}_{13}S^{A}_{21}-s^{A}_{23}S^{A}_{11}}{(S^{-1}s)_{11}}e^{\theta_{32}} & \ldots \\
(\mu_{3})_{21} & (\mu_{2})_{21}\frac{s^{A}_{23}S^{A}_{31}-s^{A}_{33}S^{A}_{21}}{(S^{-1}s)_{11}}e^{-\theta_{21}}+(\mu_{2})_{22}\frac{s^{A}_{33}S^{A}_{11}-s^{A}_{13}S^{A}_{31}}{(S^{-1}s)_{11}}+(\mu_{2})_{23}\frac{s^{A}_{13}S^{A}_{21}-s^{A}_{23}S^{A}_{11}}{(S^{-1}s)_{11}}e^{\theta_{32}} & \ldots \\
(\mu_{3})_{31} & (\mu_{2})_{31}\frac{s^{A}_{23}S^{A}_{31}-s^{A}_{33}S^{A}_{21}}{(S^{-1}s)_{11}}e^{-\theta_{21}}+(\mu_{2})_{32}\frac{s^{A}_{33}S^{A}_{11}-s^{A}_{13}S^{A}_{31}}{(S^{-1}s)_{11}}+(\mu_{2})_{33}\frac{s^{A}_{13}S^{A}_{21}-s^{A}_{23}S^{A}_{11}}{(S^{-1}s)_{11}}e^{\theta_{32}} & \ldots 
\end{array}\right. \\
& \hspace{3cm} \left.  \begin{array}{l l}
\ldots & (\mu_{2})_{11}\frac{-S^{A}_{31}e^{-\theta_{31}}}{s^{A}_{33}S^{A}_{11}-s^{A}_{13}S^{A}_{31}} + (\mu_{2})_{13}\frac{S^{A}_{11}}{s^{A}_{33}S^{A}_{11}-s^{A}_{13}S^{A}_{31}} \\
\ldots & (\mu_{2})_{21}\frac{-S^{A}_{31}e^{-\theta_{31}}}{s^{A}_{33}S^{A}_{11}-s^{A}_{13}S^{A}_{31}} + (\mu_{2})_{23}\frac{S^{A}_{11}}{s^{A}_{33}S^{A}_{11}-s^{A}_{13}S^{A}_{31}} \\
\ldots & (\mu_{2})_{31}\frac{-S^{A}_{31}e^{-\theta_{31}}}{s^{A}_{33}S^{A}_{11}-s^{A}_{13}S^{A}_{31}} + (\mu_{2})_{33}\frac{S^{A}_{11}}{s^{A}_{33}S^{A}_{11}-s^{A}_{13}S^{A}_{31}}
\end{array} \right),
\end{align*}
for all $x \geq 0$, $t\in [0,T]$, and $k \in \bar{D}_2 \setminus \mathcal{Q}$.
\end{lemma}
\begin{proof}
We only prove the relation for $M_{1}$; the proof for $M_{2}$ being similar.
Suppose first that $u_0, v_0, \tilde{u}_{0}, \tilde{u}_{1}, \tilde{u}_{2}, \tilde{v}_0$ are compactly supported. Then Lemma \ref{Snexplicitlemma} implies
\begin{align*}
\begin{cases}
 [M_1(x,t,k)]_1 = [\mu_{3}(x,t,k)]_1, \\
 [M_1(x,t,k)]_3 = \frac{[\mu_{2}(x,t,k)]_3}{s^{A}_{33}(k)}, \\
 [M_1(x,t,k)]_2 = \big( [\mu_{2}]_{1}\frac{s^{A}_{23}S^{A}_{31}-s^{A}_{33}S^{A}_{21}}{(S^{-1}s)_{11}}e^{-\theta_{21}}+[\mu_{2}]_{2}\frac{s^{A}_{33}S^{A}_{11}-s^{A}_{13}S^{A}_{31}}{(S^{-1}s)_{11}}+[\mu_{2}]_{13}\frac{s^{A}_{13}S^{A}_{21}-s^{A}_{23}S^{A}_{11}}{(S^{-1}s)_{11}}e^{\theta_{32}} \big)(x,t,k).
\end{cases}
\end{align*}
This finishes the proof of the first equation in the case of compactly supported initial data. All quantities in \eqref{M1 in terms of spectral functions} remain well-defined for $x \geq 0$, $t\in [0,T]$ and $k\in \overline{D} _{1}\setminus\{0\}$ in the case of non-compactly supported initial data. Hence, for general  $u_0, v_0 \in \mathcal{S}(\R_{+})$ and $\tilde{u}_{0}, \tilde{u}_{1}, \tilde{u}_{2}, \tilde{v}_0 \in C^{\infty}([0,T])$, formula \eqref{M1 in terms of spectral functions} can be proved using cutoff functions and taking limits (in a similar way as in \cite[Lemma 4.5]{CLgoodboussinesq}).
\end{proof}

Lemma \ref{M1XYlemma} shows that if $u_0, v_0, \tilde{u}_{0}, \tilde{u}_{1}, \tilde{u}_{2}, \tilde{v}_0$ satisfy Assumption \ref{solitonlessassumption}, then $M$ has no singularities apart from $k=0$. In this case, we can define the value of $M(x,t,k)$ at any point $k_j \in \mathcal{Q} \cap \bar{D}_n  \setminus \{0\}$ by continuity:
\begin{align}\label{Mnkjdef}
M_n(x,t,k_j) = \lim_{\underset{k \in \bar{D}_n \setminus \mathcal{Q}}{k\to k_j}} M_n(x,t,k).
\end{align}
As a consequence, we can replace $\mathcal{Q}$ with $\{0\}$ in all the above results.

\begin{lemma}\label{QtildeQlemma}
Suppose $u_0, v_0 \in \mathcal{S}(\R_{+})$ and $\tilde{u}_{0}, \tilde{u}_{1}, \tilde{u}_{2}, \tilde{v}_0 \in C^{\infty}([0,T])$ are such that Assumption \ref{solitonlessassumption} holds. Then the statements of Proposition \ref{Mnprop} and Lemmas \ref{Snexplicitlemma}-\ref{Mjumplemma} hold with $\mathcal{Q}$ replaced by $\{0\}$.
\end{lemma}

\begin{lemma}[Asymptotics of $M$ as $k \to 0$]\label{Mat1lemma}
Suppose $u_0, v_0 \in \mathcal{S}(\R_{+})$ and $\tilde{u}_{0}, \tilde{u}_{1}, \tilde{u}_{2}, \tilde{v}_0 \in C^{\infty}([0,T])$ are such that Assumptions \ref{solitonlessassumption} and \ref{originassumption} hold.
Let $p \geq 1$ be an integer.
Then there are $3 \times 3$-matrix valued functions $\{\mathcal{M}_n^{(l)}(x,t)\}$, $n = 1,\dots,12$, $l = -2,-1,0, \dots, p$, with the following properties:
\begin{enumerate}[$(a)$]
\item The function $M$ satisfies, for $x \in \R_{+}$ and $t\in [0,T]$,
\begin{align*}
& \bigg|M_n(x,t,k) - \sum_{l=-2}^p \mathcal{M}_n^{(l)}(x,t)k^l\bigg| \leq C
|k|^{p+1}, \qquad |k| \leq \frac{1}{2}, \ k \in \bar{D}_n.
\end{align*}
\item For each $n$ and each $l \geq -2$, $\{\mathcal{M}_n^{(l)}(x,t)\}$ is a smooth function of $x \in \R_{+}$ and $t\in [0,T]$.
\item For $n=1$ and $n=2$, the first coefficients are given by
\begin{align*}
\mathcal{M}_{1}^{(-2)}(x,t) & = \mathcal{M}_{2}^{(-2)}(x,t) = \alpha_{3}(x,t) \begin{pmatrix}
\omega & 0 & 0 \\
\omega & 0 & 0 \\
\omega & 0 & 0
\end{pmatrix}, \\
\mathcal{M}_{1}^{(-1)}(x,t) & = \mathcal{M}_{2}^{(-1)}(x,t) = \beta_{3}(x,t) \begin{pmatrix}
\omega^{2} & 0 & 0 \\
\omega^{2} & 0 & 0 \\
\omega^{2} & 0 & 0 
\end{pmatrix} \\
& + \gamma_{3}(x,t) \begin{pmatrix}
\omega^{2} & 0 & 0 \\
1 & 0 & 0 \\
\omega & 0 & 0
\end{pmatrix} + \delta(x,t) \begin{pmatrix}
0 & 1-\omega & 0 \\
0 & 1-\omega & 0 \\
0 & 1-\omega & 0
\end{pmatrix}, 
\end{align*}
where we omit the long expression for $\delta(x,t)$, and the third columns of $\mathcal{M}_{1}^{(0)}(x,t)$ and $\mathcal{M}_{2}^{(0)}(x,t)$ are given by
\begin{align*}
& \mathcal{M}_{1}^{(0)}(x,t) = \begin{pmatrix}
\star & \star & \epsilon_{1}(x,t) \\
\star & \star & \epsilon_{1}(x,t) \\
\star & \star & \epsilon_{1}(x,t) 
\end{pmatrix},  \quad \epsilon_{1}(x,t) := \frac{\alpha_{2}(x,t)}{\mathfrak{s}^{A(-2)}}, \\
& \mathcal{M}_{2}^{(0)}(x,t) = \begin{pmatrix}
\star & \star & \epsilon_{2}(x,t) \\
\star & \star & \epsilon_{2}(x,t) \\
\star & \star & \epsilon(_{2}x,t) 
\end{pmatrix},  \quad \epsilon_{2}(x,t) := \frac{\mathscr{S}^{A(-1)}_{2}-i\sqrt{3} \mathscr{S}^{A(-2)}(x \alpha_{2}(x,t)+\beta_{2}(x,t))}{\mathscr{S}^{A(-1)}_{2}\mathfrak{s}^{A(-2)}},
\end{align*}
where $\star$ denotes an unspecified entry.

\item For each $x \in \mathbb{R}_{+}$ and $t\in [0,T]$, the function $k \mapsto \begin{pmatrix}
\omega & \omega^{2} & 1
\end{pmatrix}M_{n}(x,t,k)$ is bounded as $k \to 0$, $k \in \bar{D}_{n}$.
\end{enumerate}
\end{lemma}

The next lemma completes the proof that $M$ satisfies RH problem \ref{RH problem for M}.

\begin{lemma}\label{Mxtjumplemma}
For each $(x,t) \in \R_{+} \times [0,T]$, $M(x,t,k)$ is an analytic function of $k \in \C \setminus \Gamma$ with continuous boundary values on $\Gamma\setminus \{0\}$. Moreover, $M$ satisfies the jump condition (\ref{Mjumpcondition}).
\end{lemma}
\begin{proof}
The analyticity and the existence of continuous boundary values are a consequence of Proposition \ref{Mnprop} and Lemma \ref{QtildeQlemma}.
By Proposition \ref{Mnprop}, $M_n$ satisfies the Lax pair equations \eqref{Xlax}. These equations imply that the functions $M_n$ are related by
$$M_{n+1}(x,t,k) = M_n(x,t,k) e^{\hat{\mathcal{L}}x + \hat{\mathcal{Z}}t}\big(M_n(0,0,k)^{-1}M_{n+1}(0,0,k)\big),$$
for $(x,t) \in \R_{+} \times [0,T]$ and $k \in \bar{D}_{n}\cap \bar{D}_{n+1}$, $n=1, \dots,12$ (with $\bar{D}_{13}:=\bar{D}_{1}$ and $M_{13}:=M_{1}$). Equation (\ref{Mjumpcondition}) now follows from Lemma \ref{Mjumplemma}.
\end{proof}

\subsection{Proof of Theorem \ref{RHth}}\label{RHsec}
Suppose $\{u(x,t), v(x,t)\}$ is a Schwartz class solution of (\ref{boussinesqsystem}) on $\R_+ \times [0, T]$ for some $T > 0$, with initial data $u_0, v_0 \in \mathcal{S}(\R_{+})$, and boundary values $\tilde{u}_{0}, \tilde{u}_{1}, \tilde{u}_{2}, \tilde{v}_0 \in C^{\infty}([0,T])$ such that Assumptions \ref{solitonlessassumption} and \ref{originassumption} hold. Define eigenfunctions $\{M_n(x,t,k)\}_{n=1}^{12}$ as in (\ref{Mndef}). Define the sectionally analytic function $M(x,t,k)$ by setting $M(x,t,k) = M_n(x,t,k)$ for $k \in D_n$. 

By Lemma \ref{Matinftylemma} and Proposition \ref{prop:first two coeff at infty}, we have
$$\lim_{k\to \infty}k\big[(M(x,t,k))_{33} - 1\big] = -\frac{2}{3}  \int_{\infty}^{x} u(x^{\prime}, t) dx'.$$
Recalling that $u(\cdot,t),v(\cdot,t)$ have rapid decay as $x \to \infty$ and that $u_t = v_x$, the formulas (\ref{recoveruv}) for $u$ and $v$ follow. On the other hand, the solution of RH problem \ref{RH problem for M} is unique (the proof of this follows as in \cite[Appendix A]{CLgoodboussinesq}). Thus it only remains to verify that $M$ satisfies RH problem \ref{RH problem for M}.

Property $(c)$ of RH problem \ref{RH problem for M} related to the asymptotics of $M$ as $k \to \infty$ follows from Lemma \ref{Matinftylemma} and Proposition \ref{prop:first two coeff at infty}.
Property $(d)$ related to the asymptotics of $M$ as $k \to 0$ follows from Lemma \ref{Mat1lemma}, property $(e)$ follows from the symmetries (\ref{Msymm}) of $M$ and Lemma \ref{QtildeQlemma}, and properties $(a)$ and $(b)$ of RH problem \ref{RH problem for M} follow from Lemma \ref{Mxtjumplemma}.

\bigskip
\noindent
{\bf Acknowledgement.} CC is a Research Associate of the Fonds de la Recherche Scientifique - FNRS. CC also acknowledges support from the European Research Council (ERC), Grant Agreement No. 101115687. JL acknowledges support from the Swedish Research Council, Grant No. 2021-03877.

\bibliographystyle{plain}

\end{document}